\documentclass[reqno]{amsart} %\documentclass[12pt,twoside,reqno]{article}
\usepackage{amsmath,amssymb}%\usepackage[fleqn]{amsmath} biao shi mo ren zuo dui qi
\usepackage{indentfirst}%shou hang suo jin liang ge dan wei
\usepackage{cite}
\usepackage{mathrsfs}
\usepackage{color}
\usepackage{mathabx}
\usepackage{dsfont}
\usepackage{marginnote}

\usepackage[top=2cm,bottom=2cm,left=2cm,right=2cm]{geometry}
\allowdisplaybreaks[1] %gong shi huan ye, neng huan dan shi jin liang bu huan
\numberwithin{equation}{section}
\newtheorem{theorem}{Theorem}[section]

\newtheorem{lemma}[theorem]{Lemma}
\newtheorem{remark}[theorem]{Remark}
\newtheorem{proposition}[theorem]{Proposition}
\newtheorem{corollary}[theorem]{Corollary}
\allowdisplaybreaks

\allowdisplaybreaks

\begin{document}
%\textheight 21  true cm \textwidth=14.5cm

%%%%%%%%%%%%%%%%%%%%%%%%%% DEF %%%%%%%%%%%%%%%%%%%%%%%%%%%%%%%%%

%\def\esssup{\mathop{\rm ess\, sup}}
%%%%%%%%%%%%%%%%%%%%%% Body of article %%%%%%%%%%%%%%%%%%%%%%%%%%%

%\baselineskip 0.65cm
\title[Global classical solutions to the ionic VPB system in a 3D periodic box]
{Global classical   solutions   to the ionic Vlasov-Poisson-Boltzmann system in a 3D periodic box}

\author[F.-C. Li]{fucai Li}
\address{School of Mathematics, Nanjing University, Nanjing
 210093, P. R. China}
\email{fli@nju.edu.cn}

\author[Y.-C. Wang]{Yichun Wang}
\address{School of Mathematics,
Southeast University, Nanjing, 211189, P. R. China}
 \email{yichunwang@seu.edu.cn}

\date{}

\begin{abstract}
% $\left\vvvert abc\right \vvvert$
We investigate the global well-posedness of the ionic Vlasov-Poisson-Boltzmann system which models the evolution of dilute collisional ions. 
This system distinguishes the electronic Vlasov-Poisson-Boltzmann system via an additional exponential nonlinearity in the coupled Poisson-Poincar\'{e}  equation, which introduces essential mathematical difficulties. 
In a three-dimensional periodic box,  
We establish the existence of a unique global-in-time classical solution with an exponential decay under small initial perturbations of a global Maxwellian that preserve mass, momentum and energy conservation laws. Our approach combines a nonlinear energy method with quantitative nonlinear elliptic estimates and new coercivity inequalities for the linearized collision operator $\mathcal{L}$ in ion dynamics.
\end{abstract}

%, nonlinear energy method}

 \maketitle

% \tableofcontents

\section{Introduction}

In plasma physics, the ionic Vlasov-Poisson-Boltzmann system describes the evolution of dilute ions through mutual interactions governed by self-consistent electrostatic field and collisions. Due to the analytical difficulties raising from the nonlinear Poisson equation, until now there has been no rigorous mathematical study on the well-posedness of this model in the $3D$ periodic box $\mathbb{T}^3=\mathbb{R}^3/\mathbb{Z}^3$. The purpose of this article is to fill this gap. First, for reader's convenience, we provide  some necessary background on the ion dynamics below.

\subsection{Background and models}
\subsubsection{Physics background}
Plasma exists commonly  in the universe as a state of matter, which manifests, for example,  in lightning, ionosphere, corona, accretion disk and interstellar medium  \cite{TK,sentis,plasma-astro}. This ionized gas comprises a multi-particle system of electrons and positively charged ions whose collective dynamics are governed by long-range electromagnetic interactions.
A characteristic example can be found in stellar atmospheric plasmas, which constitute electrons, protons, helium ions and few heavy ions \cite{Phillips}.
From a theoretical perspective, plasma kinetic theory occupies a unique position in statistical mechanics, bridging the gap between microscopic single-particle orbital descriptions and macroscopic magnetohydrodynamic formulations.

The collisions between charged particles constitute a form of electromagnetic interaction. It is commonly assumed that each charged particle interacts with surrounding particles within a spherical region defined by the characteristic Debye length scale, a phenomenon recognized as the Coulomb collision. Within this framework, the collision frequencies for electron-electron (ee), electron-ion (ei), and ion-ion (ii) interactions can be quantitatively described as  (for example, see  \cite{plasma-astro}):
\begin{equation}\label{collision-frequency}
\left\{ \begin{aligned}
&f_{\rm{ee}}=\frac{n_e e^4 \ln \Lambda}{3\sqrt{6}\pi \varepsilon_0^2 m_e^{\frac{1}{2}}(k_B T_e)^{\frac{3}{2}}},\\
&f_{\rm{ei}}=\frac{n_e Z_i^2 e^4 \ln \Lambda}{6\sqrt{3}\pi \varepsilon_0^2 m_e^{\frac{1}{2}}(k_B T_e)^{\frac{3}{2}}},\\
&f_{\rm{ii}}=\frac{n_i Z_i^4 e^4 \ln \Lambda}{3\sqrt{6}\pi \varepsilon_0^2 m_i^{\frac{1}{2}}(k_B T_i)^{\frac{3}{2}}}.
\end{aligned}
\right.
\end{equation}
Here, for the electrons,  $n_e$ is the number density, $e$ and $m_e$ are the magnitude of the charge and mass, and $T_e$ is the temperature. Similarly, we have $n_i$, $Z_i e$, $m_i$ and $T_i$ for the (one-species) ions. The parameter  $\varepsilon_0>0$ is the vacuum permittivity, $k_B>0$ is the Boltzmann constant and $\ln \Lambda$ is the Coulomb logarithm.

From \eqref{collision-frequency}, 
we see that the relaxation time of the ions is much longer than that of the electrons (one also refers to (2.12) in Cordier-Grenier \cite{Cordier-Grenier}).
On the ion's timescale, the electrons are at the local thermodynamical equilibrium and the density $n_e$ obeys the well-known Maxwell-Boltzmann relation:
\begin{align*}
  n_e=n_{e0}\exp\left\{ \frac{e\phi}{k_B T_e}\right\},
\end{align*}
where $\phi$ is the electrostatic potential and $n_{e0}$ is the number density of electrons as $\phi=0$. Then $\phi$ satisfies the Poisson equation
\begin{align*}
  \varepsilon_0 \Delta \phi=-(\rho_i+\rho_e)=en_{e0}\exp\left\{ \frac{e\phi}{k_B T_e}\right\}-\rho_i,
\end{align*}
where $\rho_e$ and $\rho_i$ are the charge density of electrons and ions respectively, and $\rho_e=-e n_e$, $\rho_i=Z_i e n_i$. Normalization of the parameters yields the so-called Poisson-Poincar\'{e} equation:
\begin{align}\label{p-p-eqn}
  \Delta \phi=e^\phi-\rho,
\end{align}
which is a semilinear second order elliptic equation for $\phi$.

\subsubsection{Models}
The ionic Vlasov-Poisson-Boltzmann system consists of a Vlasov-Boltzmann equation whose unknown is the   distribution function $F(t,x,v)$ of ions and  a Poisson-Poincar\'{e} equation whose unknown is the self-consistent potential $\phi(t,x)$, reading as
\begin{equation}\label{ivpb}
\left\{ \begin{aligned}
&\partial_t F + v \cdot \nabla_x F + E \cdot \nabla_v F = Q(F, F) \quad \mathrm{in}\quad  (0,\infty)\times \mathbb{T}^3\times \mathbb{R}^3,\\
&E=-\nabla \phi, \quad \Delta \phi=e^\phi-\rho \quad \mathrm{in}\quad  (0,\infty)\times \mathbb{T}^3,\\
&F|_{t=0}=F_0 \quad \mathrm{in}\quad  \mathbb{T}^3\times \mathbb{R}^3,
\end{aligned}
\right.
\end{equation}
where the  position $x\in \mathbb{T}^3$ and the velocity $v \in \mathbb{R}^3$ at time $t>0$, and $\rho=\int_{\mathbb{R}^3}Fdv$ is the macroscopic density of ions.
The initial distribution $F_0$ is assumed to be nonnegative.
The Boltzmann collision term $Q(F, F)$  on the right-hand side of \eqref{ivpb}$_1$ is given as the bilinear form
\begin{align*}
Q(F_1, F_2)\,&=   \int\!\! \int_{\mathbb{R}^3\times\mathbb{S}^2}
|(v-u)\cdot \omega| [F_1(u')F_2(v')-F_1(u)F_2(v)]d\omega du\\
\,&=: Q_{\mathrm{gain}}(F_1, F_2)-Q_{\mathrm{loss}}(F_1, F_2),
\end{align*}
which is the hard-sphere collision operator.
Here,  the relationship between the pre-collision velocity pair $(u, v)$ of two particles and the post-collision velocity pair $(u', v')$ obeys
\begin{align*}
u'=u+[(v-u)\cdot \omega]\omega, \quad v'=v-[(v-u)\cdot \omega]\omega,
\end{align*}
for $\omega \in \mathbb{S}^2$, which can be determined by the conservation laws of momentum and energy:
\begin{align*}
u+v=u'+v', \quad |u|^2+|v|^2 = |u'|^2+|v'|^2.
\end{align*}

Formally, Bardos et al. \cite{Bardos} derived a variant of 
the system \eqref{ivpb} with a time-varying electron temperature parameter determined by the conservation of energy, by taking the massless electron limit $\sqrt{ \frac{m_e}{m_i}}\rightarrow 0$ to the following bipolar Vlasov-Poisson-Boltzmann  system
\begin{equation*}
\left\{ \begin{aligned}
&\partial_t f_{-}+w\cdot \nabla_x f_{-}-\frac{q_e}{m_e}E\cdot \nabla_w f_{-}=Q_{-}(f_{-},f_{-}),\\
 &\partial_t f_{+}+v\cdot \nabla_x f_{+}-\frac{q_i}{m_i}E\cdot \nabla_v f_{+}=Q_{+}(f_{+},f_{+}),\\
 &\varepsilon_0 \nabla \cdot E = \int_{\mathbb{R}^3}\{q_if_{+}(v)-q_ef_{-}(v)\}dv.
\end{aligned}
\right.
\end{equation*}
In this article, we only consider the constant temperature case for presentation simplicity, as in \eqref{ivpb}.

For the classical solutions to the system \eqref{ivpb}, the following conservation laws of mass, momentum and energy hold:
\begin{gather*}
   \frac{d}{dt}\int\!\! \int_{\mathbb{T}^3\times \mathbb{R}^3} F(t,x,v)dxdv=0,\\
   \frac{d}{dt}\int\!\! \int_{\mathbb{T}^3\times \mathbb{R}^3} vF(t,x,v)dxdv=0,\\
   \frac{d}{dt}\left\{\int\!\! \int_{\mathbb{T}^3\times \mathbb{R}^3} |v|^2F(t,x,v)dxdv+\int_{\mathbb{T}^3}2\phi e^\phi+|\nabla\phi|^2 dx\right\}=0.
\end{gather*}

\subsection{Relevant literature}
Over the past several decades,
there has been a significant amount of mathematical advances on the well-posedness theory of plasma kinetic equations of Vlasov-Poisson type, mostly focusing on the electron dynamics. On the Vlasov-Poisson system for electrons, a series of works \cite{Schaeffer, Rendall, Perthame,Pfaffelmoser,Horst} proved the global existence and uniqueness of the classical solutions in domains without boundary,  \cite{Hwang2009,Hwang2010,Iacobelli} studied the same problem considering the boundary effects, and  \cite{Degond,Rendall,Smulevici,Mouhot-Villani} investigated the long time behaviors of the classical solutions.

Taking the collisions between electrons into account, the  Vlasov-Poisson-Boltzmann  system for electrons has been extensively studied. Basically, in the $L^2$ framework near the equilibrium, Guo \cite{Guo2002}, Yang-Zhao \cite{Zhao2006} and Yang-Yu-Zhao \cite{Hongjun} established the classical solutions in $\mathbb{T}^3$ and $\mathbb{R}^3$, respectively. Duan-Strain \cite{Duan2011} studied the optimal time decay of classical solutions in $\mathbb{R}^3$. Li-Yang-Zhong \cite{Zhong2020, Zhong2021} used the spectrum analysis and Green's function to investigate the refined properties of the classical solutions in $\mathbb{R}^3$. The above works are confined in hard-sphere collision kernel case (\cite{Zhong2021} also includes hard potential case). Meanwhile, for the hard and soft potentials with angular cutoff,
Duan-Yang-Zhao \cite{Duan2012, Duan2013} first dealt with the hard potentials and moderately soft potentials cases,
and then Xiao-Xiong-Zhao \cite{Zhao2017} established the well-posedness result in the case of very soft potentials,  
see also \cite{xxz-14} for the  non-cutoff hard potentials case.
If the interactions between particles and boundaries are  considered simultaneously, the $L^2$-$L^\infty$ method
developed by Guo \cite{Guo2010} is involved. In term of this framework,
Cao-Kim-Lee \cite{Kim} constructed global strong solutions for hard-sphere case with homogenous Neumann condition for the electric potential and the diffuse boundary condition for the particle distribution functions. See also  \cite{fucai} for soft potentials case with the incoming  boundary condition.
For the grazing collision limit from the Vlasov-Poisson-Boltzmann system to the Vlasov-Poisson-Landau system, one refers to He-Lei-Zhou \cite{He-Lei-Zhou}. 
In addition, there are many contributions  on the  hydrodynamic limit of the electronic Vlasov-Poisson-Boltzmann system, interested reader can refer to, for example, \cite{Juhi,LW-23,LYZ-21, ZhouFujun-SIAM-2021}.

In contrast, for the Vlasov-Poisson type equations of ion dynamics, the mathematical research is scarce due to the additional exponential nonlinearity in the Poisson-Poincar\'{e} equation \eqref{p-p-eqn}, which is the main difference between the ionic and electronic Vlasov-Poisson type systems. Bouchut \cite{Bouchut} first constructed the global weak solutions to the ionic Vlasov-Poisson system in $\mathbb{R}^3$. In addition, Bardos, et al. \cite{Bardos} proved the existence and uniqueness of global weak solutions with the conservation of energy to the ionic Vlasov-Poisson system in both a bounded domain $\Omega\subset \mathbb{R}^2$ and $\mathbb{T}^3$. For classical solutions, global well-posedness results of the ionic Vlasov-Poisson system have been established in $\mathbb{T}^2$ and in $\mathbb{T}^3$ 
\cite{Iacobelli2021b}, and in $\mathbb{R}^3$ \cite{Iacobelli2021a} by Griffin-Pickering and Iacobelli, and in a bounded domain $\Omega \subset \mathbb{R}^3$ by Cesbron-Iacobelli \cite{Iacobelli}. 
Recently, Choi-Koo-Song \cite{Choi-Koo-Song} established the global existence of Lagrangian solutions to the ionic Vlasov-Poisson system in $\mathbb{T}^d\, (d\geq 1)$. In addition, Gagnebin-Iacobelli \cite{Gagnebin-Iacobelli} studied the nonlinear Landau damping of the ionic Vlasov-Poisson system in $\mathbb{T}^d\, (d\geq 1)$.

If taking the ion-ion collisions into account, rigorous well-posedness results  are limited. However, notice that
Li-Yang-Zhong \cite{Zhong2016} obtained the optimal decay rate of the ionic Vlasov-Poisson-Boltzmann system in $\mathbb{R}^3$ near Maxwellians, by using spectrum analysis and macro-micro decomposition. 
Flynn \cite{Flynn} proved the local well-posedness  of the ionic Vlasov-Poisson-Landau system in $\mathbb{T}^3$ for large initial data.

To the best of our knowledge, in $\mathbb{T}^3$, the global well-posedness result on the ionic Vlasov-Poisson-Boltzmann system \eqref{ivpb}$_1$--\eqref{ivpb}$_2$ is missing from literature. The goal of this article is to establish the global  classical solutions to the  problem \eqref{ivpb} near Maxwellians in $\mathbb{T}^3$ with an exponential decay, which will be proposed in what follows.

\subsection{Perturbed equations near the equilibrium}

In this article, we shall construct a classical solution $F(t,x,v)$ to the ionic  Vlasov-Poisson-Boltzmann  system \eqref{ivpb}$_1$--\eqref{ivpb}$_2$ near the global Maxwellian
\begin{align*}
\mu\equiv \mu(v)=\frac{1}{(2\pi)^{\frac{3}{2}}}\exp\left\{-\frac{|v|^2}{2}\right\}.
\end{align*}
For this purpose, we define the standard perturbation $f(t,x,v)$ to $\mu$ as
 $$F= \mu+\sqrt{\mu} f,$$
and the initial datum $f_0(x,v)$ as $F_0=\mu+\sqrt{\mu}f_0\geq 0$.

We rewrite the problem \eqref{ivpb} by using the unknown  $f$. Thus, the ionic  Vlasov-Poisson-Boltzmann   system now becomes
\begin{align} \label{perturb eqn}
&\left\{\partial_t + v\cdot \nabla_x + E \cdot \nabla_v \right\}f-E\cdot v\sqrt{\mu}-\frac{v}{2}\cdot E f
=-\mathcal{L}f+ \Gamma(f, f),
\end{align}
coupled with
\begin{equation}\label{poisson}
\left\{ \begin{aligned}
&E=-\nabla \phi,\\
&\Delta \phi=e^\phi-\int_{\mathbb{R}^3}\sqrt{\mu}f dv-1,
\end{aligned}
\right.
\end{equation}
and supplemented with the initial datum
\begin{align}\label{initial-condition}
  f(0,x,v)=f_0(x,v).
\end{align}
Here, $\mathcal{L}$ is the linearized Boltzmann collision operator defined as
\begin{align*}
\mathcal{L}f=\nu f-\mathcal{K}f,
\end{align*}
where $\nu(v)= \int_{\mathbb{R}^3} |(v-u)\cdot \omega| \mu(u)dud\omega$, and $\mathcal{K}$ is split into $\mathcal{K}:=\mathcal{K}_2-\mathcal{K}_1$ with
\begin{align*}
\mathcal{K}_1 f(v):=&  \int\!\! \int_{\mathbb{R}^3\times\mathbb{S}^2} |(v-u)\cdot \omega| \mu^{1/2}(u)\mu^{1/2}(v)f(u)dud\omega,\\
\mathcal{K}_2 f(v):=&  \int\!\! \int_{\mathbb{R}^3\times\mathbb{S}^2} |(v-u)\cdot \omega| \mu^{1/2}(u)\{\mu^{1/2}(u')g(v')+\mu^{1/2}(v')g(u')\}
dud\omega.
\end{align*}
We also define the bilinear collision operator $\Gamma (\cdot,\cdot)$ as
\begin{align*}
\Gamma(g_1,g_2)(v):=\,&\mu^{-1/2}(v)Q(\mu^{1/2}g_1, \mu^{1/2} g_2)(v)\\
=\,& \int\!\! \int_{\mathbb{R}^3\times\mathbb{S}^2}|(v-u)\cdot \omega|\mu^{1/2}(u)g_1(u')g_2(v')dud\omega\\
&-\int\!\! \int_{\mathbb{R}^3\times\mathbb{S}^2}|(v-u)\cdot \omega|\mu^{1/2}(u)g_1(u)g_2(v)dud\omega\\
=:\,& \Gamma_{\rm{gain}}(g_1,g_2)(v)
-\Gamma_{\rm{loss}}(g_1,g_2)(v).
\end{align*}

By assuming that $[F_0,\phi_{F_0}]$ has the same mass, momentum and total energy as the steady state $[\mu, 0]$, we can then rewrite the conservation laws in terms of $[f,\phi_f]$ as
\begin{gather}
   \int\!\! \int_{\mathbb{T}^3\times\mathbb{R}^3}f(t,x,v)\sqrt{\mu(v)}dxdv\equiv 0,\label{mass}\\
  \int\!\! \int_{\mathbb{T}^3\times\mathbb{R}^3}vf(t,x,v)\sqrt{\mu(v)}dxdv\equiv 0,\label{momentum}\\
   \int\!\!  \int_{\mathbb{T}^3\times \mathbb{R}^3}\frac{1}{2}|v|^2 f(t,x,v)\sqrt{\mu(v)}dxdv+\int_{\mathbb{T}^3}\phi_f e^{\phi_f}dx+\int_{\mathbb{T}^3}\frac{1}{2} |\nabla \phi_f|^2dx  \equiv 0. \label{energy}
\end{gather}
Here we have used the notations $\phi_{F_0}$ and $\phi_{f}$ to emphasise that $\phi$ depends on ${F_0}$ and ${f}$, respectively. As a by-product, we also deduce that
\begin{align}\label{neutural}
  \int_{\mathbb{T}^3}e^{\phi_f}dx=\int_{\mathbb{T}^3} (\Delta \phi_f +1)dx + \int\!\! \int_{\mathbb{T}^3\times\mathbb{R}^3}f(t,x,v)\sqrt{\mu(v)}dxdv=|\mathbb{T}^3|=1.
\end{align}

\subsection{Notations and energy functional}
Throughout this article,
the constant $C>0$ in the relationship $a\leq Cb$ may be different and depends on a variable collection of parameters in the context. This notation is used mostly inside the proofs of lemmas, propositions and theorems.
The symbol $a\sim b$ means that there exist two constants  $c, C>0$ such that $cb\leq a\leq Cb$.

We use $(\cdot , \cdot)$ to denote the standard $L^2$ inner product in $\mathbb{T}^3_x\times \mathbb{R}^3_v$ for a pair of functions $g, h\in L^2(\mathbb{T}^3_x\times \mathbb{R}^3_v)$ and $\langle\cdot,\cdot\rangle$ to denote the $L^2$ product in $\mathbb{R}^3_v$. In the sequel, we use $\|\cdot\|$ to denote $L^2$ norms in $\mathbb{T}^3_x\times \mathbb{R}^3_v$ or $\mathbb{T}^3_x$, and $|\cdot|_0$ to denote $L^2$ norm in $\mathbb{R}^3_v$ unless otherwise specified.
We also denote $$( h_1, h_2)_\nu:= ( \nu(v)h_1, h_2)$$ to be the $\nu$-weighted $L^2$ inner product in $\mathbb{T}_x^3\times \mathbb{R}_v^3$, and $\|\cdot\|_\nu$ to be the corresponding weighted $L^2$ norm. While $|\cdot|_\nu$ denotes the weighted $L^2$ norm in $\mathbb{R}^3_v$.

Let the multi-indices $\gamma$ and $\beta$ be
\begin{align*}
 \gamma=[\gamma_0,\gamma_1,\gamma_2,\gamma_3],\quad \beta=[\beta_1, \beta_2, \beta_3],
\end{align*}
and
\begin{align*}
 \partial_\beta^\gamma :=\partial_t^{\gamma_0}\partial_{x_1}^{\gamma_1}\partial_{x_2}^{\gamma_2}\partial_{x_3}^{\gamma_3}
 \partial_{v_1}^{\beta_1}\partial_{v_2}^{\beta_2}\partial_{v_3}^{\beta_3}.
\end{align*}
For the two multi-indices $\tilde{\alpha}$ and $\alpha$, we denote
$\tilde \alpha \leq \alpha$ if each component of $\tilde \alpha $ is not greater than that of $\alpha$'s, while $\tilde \alpha < \alpha$ means that $\tilde \alpha\leq \alpha$ and $|\tilde \alpha |< |\alpha|$.

For a fixed $N\in \mathbb{N}$, we define
\begin{align*}
  \vvvert f \vvvert (t):=& \sum_{|\gamma|+|\beta|\leq N} \|\partial_\beta^\gamma f(t)\|,\\
  \vvvert f \vvvert_\nu(t):=& \sum_{|\gamma|+|\beta|\leq N} \|\partial_\beta^\gamma f(t)\|_\nu,
\end{align*}
where $\nu(v)\sim (1+|v|^2)^{1/2}$.
We also define the high order energy functional for $f(t,x,v)$ as
\begin{align}\label{Efunctional}
  \mathcal{E}(f)(t)=\mathcal{E}(t):=\vvvert f \vvvert   ^2 (t)+\int_0^t \vvvert f \vvvert_{\nu}^2 (s)ds.
\end{align}
Given the initial datum $f_0(x,v)$, we define the initial energy functional:
\begin{align*}
  \mathcal{E}(0):=\sum_{|\gamma|+|\beta|\leq N} \|\partial_\beta^\gamma f_0\|^2\triangleq \vvvert    f_0\vvvert   ^2,
\end{align*}
where the temporal derivatives of $f_0$ are defined naturally through \eqref{perturb eqn}--\eqref{poisson}.

\subsection{Main result}
We state our main result as follows.
\begin{theorem}\label{Thm1}
Let $N\geq 4$ be an integer. Assume that $F_0(x,v)=\mu(v)+\sqrt{\mu(v)}f_0(x,v)\geq 0$ and $f_0$ satisfies the conservation laws \eqref{mass}--\eqref{energy}.   Then, there exists a  $\varkappa >0$ such that if
$$\mathcal{E}(0)\leq \varkappa,$$
 the ionic  Vlasov-Poisson-Boltzmann  system \eqref{perturb eqn}--\eqref{poisson} with the initial datum \eqref{initial-condition} enjoys a unique global classical solution $f(t,x,v)$
satisfying
$F(t,x,v)=\mu(v)+\sqrt{\mu(v)}f(t,x,v)\geq 0$, and
\begin{align*}
  \sup_{0\leq t \leq \infty}\mathcal{E}(t)\leq C_0 \mathcal{E}(0)
\end{align*}
for some constant $C_0>0$. Moreover, there exist two constants $C_1, C_2>0$ such that
\begin{align*}
 \vvvert f \vvvert^2(t)\leq C_1 e^{-C_2t}\mathcal{E}(0).
\end{align*}
\end{theorem}

\subsection{Key points of the proof}
Now we explain our main ideas and key points developed to the proof of  Theorem \ref{Thm1}.

\subsubsection{Nonlinear energy method}
Our main strategy aimed to the nonlinear stability of the global Maxwellian $\mu$ for the ionic  Vlasov-Poisson-Boltzmann  system
\eqref{perturb eqn}--\eqref{poisson} is based on the nonlinear $L^2$ energy method first devised by Guo \cite{Guo2002L}. It has inspired many subsequent works, such as \cite{Guo2002} for the electronic  Vlasov-Poisson-Boltzmann  system, \cite{Guo2003} for the Vlasov-Maxwell-Boltzmann system, and \cite{Guo2012,JAMS2023} for the Vlasov-Poisson-Landau
system.

This method can be applied to many kinetic models of the form
\begin{align}\label{toy-model}
  \partial_t f+v\cdot \nabla_x f+\mathcal{L}f=\Gamma (f,f).
\end{align}
The main idea is to find an energy norm $\|\cdot \|_{\mathcal{E}}$, a dissipation norm $\|\cdot \|_{\mathcal{D}}$ and a suitable constant $\delta>0$ such that
\begin{align}\label{dissipation}
  \frac{d}{dt}\|f\|^2_{\mathcal{E}}+\delta \|f\|^2_{\mathcal{D}}\leq 0.
\end{align}
The choice of the energy and dissipation norms is delicate.  For the model \eqref{perturb eqn} in hard-sphere case, the energy norm is actually $\vvvert \cdot \vvvert$, and the dissipation norm is actually $\vvvert \cdot \vvvert_{\nu}$
which comes from the linear part of the Boltzmann collision operator $\mathcal{L}f=\nu f-\mathcal{K}f$, and satisfies
$|\langle f, \Gamma (f,f)\rangle_{\mathcal{E}}|\leq C
\|f \|_{\mathcal{E}} \|f \|^2_{\mathcal{D}}$ by Lemma \ref{Gamma}. Then \eqref{dissipation} is obtained from a suitable small assumption of $\|f \|_{\mathcal{E}}$ in Section \ref{decay-section}, combining with the positivity of $\mathcal{L}$ in Section \ref{coerci-section}. Moreover, noting that the norm $\|\cdot \|_{\mathcal{E}}$ is weaker than $\|\cdot \|_{\mathcal{D}}$, then the solution $f$ decays exponentially in $\|\cdot \|_{\mathcal{E}}$   by \eqref{dissipation} and Gronwall's inequality.

\subsubsection{Elliptic estimates of the  Poisson-Poincar\'{e} equation \eqref{p-p-eqn}}
To study the Poisson-Poincar\'{e} equation \eqref{p-p-eqn} in $\mathbb{T}^3$, we follow \cite{Iacobelli2021b} to decompose the electric potential $\phi$ into the form $\phi=\overline{\phi}+\widehat{\phi}$ satisfying
\begin{align*}
\Delta \overline{\phi}=1-\rho,\quad \Delta \widehat{\phi}=e^{\overline{\phi}+\widehat{\phi}}-1.
\end{align*}
Assuming $$\rho:=\int_{\mathbb{R}^3} \sqrt{\mu}f dv+1\in L^2(\mathbb{T}^3),$$
we obtain the existence and uniqueness of $\overline{\phi}\in H^2(\mathbb{T}^3)$ and $\widehat{\phi}\in H^2(\mathbb{T}^3)$ with upper bounds independent of $\overline{\phi}$ and $\widehat{\phi}$.
It follows that $\phi\in C(\mathbb{T}^3)$ with $\|\phi\|_{C(\mathbb{T}^3)}=O(1)$ (see Remark \ref{remark3.2}) and thus $e^\phi$ has  positive upper and lower bounds. By the mean value theorem, we obtain in \eqref{U-xi} that
\begin{align*}
  \Delta \phi=e^\phi-1+1-\rho=\phi e^{\phi_\xi}-\int_{\mathbb{R}^3} \sqrt{\mu} fdv,
\end{align*}
hence we can use the standard linear elliptic estimates to get
\begin{align*}
  \|\phi\|_{L^\infty(\mathbb{T}^3)} \leq C\| \phi\|_{H^2(\mathbb{T}^3)}\leq C e^{2\|\phi\|_{L^\infty(\mathbb{T}^3)}}\|f\|_{L^2(\mathbb{T}^3\times \mathbb{R}^3)}.
\end{align*}
Since $\|f\|_{L^2(\mathbb{T}^3\times \mathbb{R}^3)}$ is controlled by $ \|f\|_{\mathcal{E}}$ which is assumed to be sufficiently small,   we know that  $\|\phi\|_{L^\infty(\mathbb{T}^3)}$ is also small. Therefore, we obtain $\| \phi\|_{H^2(\mathbb{T}^3)}\leq C \|f\|_{L^2(\mathbb{T}^3\times \mathbb{R}^3)}$, where $C$ is independent of $\phi$ and $f$. The stability of \eqref{p-p-eqn} follows by Corollary \ref{U1-U2}.

Furthermore, taking $\partial^\gamma$ derivatives to \eqref{p-p-eqn} yields
\begin{align*}
  \Delta \partial^\gamma \phi-e^\phi \partial^\gamma \phi= \mathcal{S}-\int_{\mathbb{R}^3} \sqrt{\mu}\partial^\gamma f dv,
\end{align*}
where $\mathcal{S}$ is constituted by lower order derivatives of $\phi$. Inductively we have (see also Remark \ref{remark3.6})
\begin{align*}
  \| \partial^\gamma \phi\|_{H^2(\mathbb{T}^3)}\leq C \sum_{|\gamma|\leq N}\|\partial^\gamma f\|_{L^2(\mathbb{T}^3\times \mathbb{R}^3)}\leq C \|f\|_{\mathcal{E}}.
\end{align*}
The above $L^2$ estimates for the  Poisson-Poincar\'{e} equation \eqref{p-p-eqn} are sufficient to establish
 the local well-posedness of the problem \eqref{ivpb}, as in Section \ref{local-section}. Meanwhile, the elliptic estimates make $\|\phi\|_{\mathcal{E}}$ be absorbed by $\|f\|_{\mathcal{E}}$. This is the reason
that we only need to define $\|\cdot\|_{\mathcal{E}}$ by $\vvvert f \vvvert $ rather than $\vvvert f \vvvert+\vvvert \phi \vvvert$. As a contrast, one refers to \cite{Guo2003} for the explanations on the  construction of energy functional for the  Vlasov-Maxwell-Boltzmann system
where the electronic and magnetic fields are included.

\subsubsection{Coercivity of the linearized operator $\mathcal{L}$}
Heuristically, we denote $\mathcal{B}=-\mathcal{L}-v\cdot \nabla_x$, then the integral form of \eqref{toy-model}
read as  (cf.  \cite{cercignani} for the Boltzmann equation case)
\begin{align*}
  f(t)=e^{t\mathcal{B}}f(0)+\int_0^t e^{(t-s)\mathcal{B}}\Gamma(f(s),f(s))ds.
\end{align*}
If it were proved that
\begin{align*}
 \|e^{t\mathcal{B}}\|\leq C e^{-Ct} \quad\text{and}\quad  \|\Gamma (f,f)\|\leq C \|f\|^2
\end{align*}
hold in some norm $\|\cdot\|$ for some $C>0$, then we   have
\begin{align*}
  \|f(t)\|\leq C e^{-Ct} \|f(0)\|+\int_0^t C e^{-C(t-s)}\|f(s)\|^2 ds.
\end{align*}
Denoting $\breve{f}=\sup_t \|f(t)e^{ct}\|$, it holds
\begin{align*}
 \breve{f}\leq C \|f(0)\|+C \breve{f}^2,
\end{align*}
which implies the boundedness of $\breve{f}$ under some smallness conditions. However, the five dimensional null space of $\mathcal{L}$ prevents us from this strategy. In fact, $e^{-t\mathcal{L}}h=h$ as  $h$ belongs to the kernel of
$ \mathcal{L}$, so $\|e^{-t\mathcal{L}}\|\leq  e^{-Ct}$ does not hold.

Henceforth, we devote to establishing the coercivity estimates of $\mathcal{L}$ in terms of the solution $f$ to the ionic  Vlasov-Poisson-Boltzmann  system \eqref{perturb eqn}--\eqref{poisson} with the conservation laws \eqref{mass}--\eqref{energy}.
The key point is to control the macroscopic part of $\partial^\gamma f$ by its microscopic part in $\|\cdot\|_\nu$ norm. Define the macroscopic function $a(t,x)$, $b(t,x)$ and $c(t,x)$ as
\begin{align*}
  \mathbf{P} \partial^\gamma f(t,x,v):=\{\partial^\gamma a(t,x)+\partial^\gamma b(t,x)\cdot v+\partial^\gamma c(t,x)|v|^2\}\sqrt{\mu},
\end{align*}
where $\mathbf{P}$ is the orthogonal projection from $L^2_v$ to $\ker (\mathcal{L}) (=:\mathscr{N})$.
Due to the nonlinear exponential term in the nonlinear Poisson equation \eqref{poisson}, it raises new difficulties for estimating the pure spatial derivatives of $a(t,x)$, which satisfies \eqref{lbi}: $$\partial^0 b_i+\partial^i a-E_i=l_{bi}+h_{bi}.$$
Here, $l_{bi}$ and $h_{bi}$ are the components of $l$ and $h$ under $v_i\sqrt{\mu}\in L^2(\mathbb{R}^3)$ defined in \eqref{def-component-l+h}.
A first glimpse leads to $\|\partial^i a\|\leq \|\partial^0 b_i\|+\|l_{bi}\|+\|h_{bi}\|+\|E_i\|$, but $\|E_i\|$ cannot be controlled by
$$C\sum_{|\gamma|\leq N}\|\{\mathbf{I}-\mathbf{P}\}\partial^\gamma f(t)\| +C\sqrt{M_0}\sum_{|\gamma|\leq N}\|\partial^\gamma f(t)\|,$$
required by \eqref{abc}, where $M_0\ll 1$. We thus ought to combine the density of ion: $\varrho_0 a +\varrho_2 c$ $(\varrho_0=\int \mu dv, \varrho_2=\int |v|^2 \mu dv)$, with the density of electron: $e^\phi$. In fact, from \eqref{lc} and \eqref{lbi} we have three identities:
\begin{equation*}
\left.%\{
 \begin{aligned}
&{(\rm i}).\,\,-\Delta a+\nabla\cdot E=\nabla\cdot \partial^0 b-\sum_i \partial^i (l_{bi}+h_{bi}), \\
& ({\rm ii}).\,\,-\Delta c=-\nabla\cdot (l_c+h_c),\\
& ({\rm iii}).\,\,-\Delta \phi-\nabla \cdot E=0.
\end{aligned}
\right.
\end{equation*}
For some sufficiently large constant $\Theta>0$, we operate (i)$\times (\rho_0 a+\rho_2 c)$+(ii)$\times \Theta c$+(iii)$\times (e^\phi-1)$ and then integrate the result  over $\mathbb{T}^3$ to get
\begin{align*}
  &\int_{\mathbb{T}^3} |\nabla a|^2 +\Theta |\nabla c|^2+|\nabla \phi|^2 e^\phi+(\nabla\cdot E)^2+(\nabla a \cdot \nabla c)dx\\
  &\qquad \leq C (\|\partial^0 b\|+\|l_{bi}+h_{bi}\|+\|l_c+h_c\|)(\|\nabla a\|+\|\nabla c\|).
\end{align*}
It follows that
\begin{align*}
  \|\nabla a\|+\|\nabla c\|\leq &\,C (\|\partial^0 b\|+\|l_{bi}+h_{bi}\|+\|l_c+h_c\|),
\end{align*}
which can be controlled by $C\sum_{|\gamma|\leq N}\|\{\mathbf{I}-\mathbf{P}\}\partial^\gamma f(t)\| +C\sqrt{M_0}\sum_{|\gamma|\leq N}\|\partial^\gamma f(t)\|$.
Throughout the above deductions, we remark that the identity  \eqref{lc} of $c$ plays a crucial role in estimating $\nabla a$. The high-order derivative analogies of this estimate are presented in Section \ref{coerci-section} with some extra techniques.

\subsection{Outline of the article}
The structure of this paper is organized as follows. Section \ref{prelimi-section} presents fundamental estimates for the Boltzmann operator, including essential proof details. In Section \ref{ellip-section}, we establish the $L^2$ regularity estimates for the Poisson-Poincaré equation \eqref{p-p-eqn}. Section \ref{local-section} focuses on demonstrating local well-posedness of the initial value problem \eqref{perturb eqn}--\eqref{initial-condition} with small initial data through Picard iteration. The coercivity property of the linearized operator $\mathcal{L}$ is rigorously proved in Section \ref{coerci-section} (Theorem \ref{coercivityThm}). Finally, Section \ref{decay-section} constructs a global-in-time solution with exponential decay properties.

\section{Preliminaries}\label{prelimi-section}
For reader's convenience, we present some basic estimates about the linearized operator $\mathcal{L}=\nu-\mathcal{K}$ and the nonlinear operator $\Gamma$
in this section.
\begin{lemma}
It holds $\langle \mathcal{L} g,h\rangle=\langle \mathcal{L} h,g\rangle$, $\langle \mathcal{L} g,g\rangle\geq 0$. And $\mathcal{L}g=0$ if and only if
$$g=\mathbf{P}g,$$
where $\mathbf{P}: L_v^2(\mathbb{R}^3)\rightarrow \mathscr{N}$ is the orthogonal projection in $L_v^2(\mathbb{R}^3)$ to the null space $\mathscr{N}$ of $\mathcal{L}$, which is a five dimensional space spanned by $\{\sqrt{\mu}, v_1\sqrt{\mu}, v_2\sqrt{\mu},v_3\sqrt{\mu}, |v|^2\sqrt{\mu}\}$.

Moreover, there exists a constant $\delta>0$ such that
\begin{align*}
  \langle \mathcal{L} g,g\rangle\geq \delta |\{\mathbf{I}-\mathbf{P}\}g|_\nu^2.
\end{align*}
\end{lemma}
For the proofs, one can refer to \cite{Guo2002,Guo-soft-2003}.
\begin{lemma}\label{K}
For the hard-sphere collision kernel and the multi-index $\beta$, the collision frequency $\nu(v)$ satisfies
\begin{align}\label{nu-derivative}
 \partial_\beta \nu(v)\leq C (1+|v|)^{1-|\beta|}.
\end{align}
In addition, for any $\eta>0$, there exists a $C_\eta>0$ such that
\begin{align}
\langle \partial_\beta [\nu g], \partial_\beta g\rangle
&\,\geq |\partial_\beta g|_\nu^2 -\eta \sum_{|\beta_1|\leq |\beta|}|\partial_{\beta_1} g|_0^2-C_\eta | g|_0^2, \,\,\,\beta\neq 0,\label{nu-coercive}\\
|\langle \partial_\beta [ \mathcal{K} g_1], \partial_\beta g_2\rangle|
&\,\leq \bigg\{\eta \sum_{|\beta_1|\leq |\beta|}|\partial_{\beta_1}g_1|_0+C_\eta |g_1|_0\bigg\}|\partial_\beta g_2|_0.\label{K-estimate}
\end{align}
\end{lemma}
\begin{proof}
The proof of \eqref{nu-derivative} can be found in   \cite{Guo2002L}.
By using \eqref{nu-derivative} and compact interpolation between $\partial_\beta$ and $\partial_0$, we have
\begin{align*}
  \langle \partial_\beta [\nu g], \partial_\beta g\rangle
  =\,&|\partial_\beta g|_\nu^2+\sum_{0< |\beta_1|\leq |\beta|}C_\beta^{\beta_1}\langle \partial_{\beta_1}\nu \partial_{\beta-\beta_1}g, \partial_\beta g\rangle\\
  \geq\, & |\partial_\beta g|_\nu^2-\eta \sum_{|\beta_1|\leq |\beta|}|\partial_{\beta_1} g|_0^2-C_\eta | g|_0^2
\end{align*}
for any $\eta>0$, thus \eqref{nu-coercive} holds. For the proof of \eqref{K-estimate},  note that
\begin{align*}
  |\partial_\beta [\mathcal{K}g]|_0\leq \eta \sum_{|\beta_1|=|\beta|}|\partial_{\beta_1}g|_0+C_\eta |g|_0,
\end{align*}
which was presented in Lemma 2.1 in \cite{Duan2012}. Thus \eqref{K-estimate} is valid.
\end{proof}

\begin{lemma}\label{Gamma}
Let the indices $\beta_0+\beta_1+\beta_2=\beta$, $\gamma_1+\gamma_2=\gamma$. It holds
\begin{align*}
  \partial_\beta^\gamma \Gamma(g_1,g_2)=\sum C_\beta^{\beta_0 \beta_1 \beta_2} C_\gamma^{\gamma_1 \gamma_2} \Gamma^0 (\partial_{\beta_1}^{\gamma_1}g_1, \partial_{\beta_2}^{\gamma_2}g_2),
\end{align*}
where
\begin{align*}
  \Gamma^0 (\partial_{\beta_1}^{\gamma_1}g_1, \partial_{\beta_2}^{\gamma_2}g_2)=\,&\int\!\!\int_{\mathbb{R}^3 \times \mathbb{S}^2} |(v-u)\cdot \omega|\partial_{\beta_0}[\sqrt{\mu(u)}]\partial_{\beta_1}^{\gamma_1}g_1(u') \partial_{\beta_2}^{\gamma_2}g_2(v')dud\omega\\
  &-\int\!\!\int_{\mathbb{R}^3 \times \mathbb{S}^2} |(v-u)\cdot \omega|\partial_{\beta_0}[\sqrt{\mu(u)}]\partial_{\beta_1}^{\gamma_1}g_1(u) \partial_{\beta_2}^{\gamma_2}g_2(v)dud\omega\\
  =:\,&\Gamma^0_{\rm{gain}}-\Gamma^0_{\rm{loss}}.
\end{align*}
Moreover, for $|\beta|+|\gamma|\leq N (N\geq 4)$, we have
\begin{align}\label{Gamma-derivative}
&\left|\left( \Gamma^0 (\partial_{\beta_1}^{\gamma_1}g_1, \partial_{\beta_2}^{\gamma_2}g_2), \partial_\beta^\gamma g_3\right)\right|\nonumber\\
&\,\quad \leq C \bigg\{ \sum_{|\beta_k|+|\gamma_k|\leq N}\|\partial_{\beta_k}^{\gamma_k}g_1\|\bigg\}
\|\partial_{\beta_2}^{\gamma_2}g_2\|_\nu \|\partial_{\beta}^{\gamma}g_3\|_\nu
+C \bigg\{ \sum_{|\beta_k|+|\gamma_k|\leq N}\|\partial_{\beta_k}^{\gamma_k}g_1\|_\nu\bigg\}
\|\partial_{\beta_2}^{\gamma_2}g_2\| \|\partial_{\beta}^{\gamma}g_3\|_\nu\nonumber\\
&\qquad\, +  C \bigg\{ \sum_{|\beta_k|+|\gamma_k|\leq N}\|\partial_{\beta_k}^{\gamma_k}g_2\|\bigg\}
\|\partial_{\beta_1}^{\gamma_1}g_1\|_\nu \|\partial_{\beta}^{\gamma}g_3\|_\nu
+C \bigg\{ \sum_{|\beta_k|+|\gamma_k|\leq N}\|\partial_{\beta_k}^{\gamma_k}g_2\|_\nu\bigg\}
\|\partial_{\beta_1}^{\gamma_1}g_1\| \|\partial_{\beta}^{\gamma}g_3\|_\nu.
\end{align}
Here, $( \cdot, \cdot)$ is the inner product in $L^2(\mathbb{T}_x^3\times \mathbb{R}^3_v)$.

%Let $g_k(x,v)$, $k=1,2,3$, be smooth functions, and $(i,j)\in \{(1,2), (2,1)\}$. Then we have
%\begin{align*}
%&|\langle \partial_\beta \Gamma (g_1,g_2),g_3\rangle|\leq C \sum_{\beta_1+\beta_2=\beta}\sum_{(i,j)}\int_{\mathbb{T}^3}
 % \left[\int \nu |\partial_{\beta_1}g_i|^2 dv\right]^{1/2}
 % \left[\int  |\partial_{\beta_2}g_j|^2 dv\right]^{1/2}
  %\left[\int \nu g_3^2 dv\right]^{1/2}
 % dx,\\
%& \left\|\int \Gamma (g_1,g_2)g_3 dv\right\|\leq C \sup_{x,v}\{|\nu^3g_3|\}\sup_x\left[\int |g_i(x,v)|^2 dv\right]^{1/2}\|g_j\|,\\
%& \|\Gamma(g_1,g_2)g_3\|\leq C \sup_{x,v}\{|\nu g_3|\}\sup_x\left[\int |g_i(x,v)|^2 dv\right]^{1/2}\|g_j\|.
%\end{align*}

\end{lemma}
\begin{remark}\label{gamma-zero}
For $|\beta|+|\gamma|=0$, it holds that $|\beta_1|+|\gamma_1|=0$ and $|\beta_2|+|\gamma_2|=0$. Then we can deduce from \eqref{Gamma-derivative} that
\begin{align*}
&\left|\left( \Gamma (g_1, g_2), g_3\right)\right|
  \leq C \bigg\{ \sum_{|\gamma|\leq 2}\|\partial^{\gamma}g_1\|\bigg\}
\|g_2\|_\nu \|g_3\|_\nu
+C \bigg\{ \sum_{|\gamma|\leq 2}\|\partial^{\gamma}g_1\|_\nu\bigg\}
\|g_2\| \|g_3\|_\nu,
\end{align*}
or
\begin{align*}
&\left|\left( \Gamma (g_1, g_2), g_3\right)\right| \leq C \bigg\{ \sum_{|\gamma|\leq 2}\|\partial^{\gamma}g_2\|\bigg\}
\|g_1\|_\nu \|g_3\|_\nu
+C \bigg\{ \sum_{|\gamma|\leq 2}\|\partial^{\gamma}g_2\|_\nu\bigg\}
\|g_1\| \|g_3\|_\nu.
\end{align*}
\end{remark}
\medskip
\begin{proof}[Proof of Lemma \ref{Gamma}]
In light of Lemma 2.3 in Guo \cite{Guo2002}, we obtain that
\begin{align}\label{G-ineqn}
&\big|\big\langle \Gamma^0 (\partial_{\beta_1}^{\gamma_1}g_1, \partial_{\beta_2}^{\gamma_2}g_2), \partial_\beta^\gamma g_3\big\rangle\big| \leq C |\partial_{\beta_1}^{\gamma_1}g_1|_\nu
|\partial_{\beta_2}^{\gamma_2}g_2|_0 |\partial_{\beta}^{\gamma}g_3|_\nu
+C |\partial_{\beta_2}^{\gamma_2}g_2|_\nu
|\partial_{\beta_1}^{\gamma_1}g_1|_0 |\partial_{\beta}^{\gamma}g_3|_\nu.
\end{align}
% Here, $\langle \cdot, \cdot\rangle$ is the inner product in $L^2(\mathbb{R}^3_v)$.
Recalling $H^2(\mathbb{T}^3)\subset L^\infty(\mathbb{T}^3)$, we have
\begin{gather*}
   \sup_x \int_{\mathbb{R}^3}|g(x,u)|^2 du\leq \int_{\mathbb{R}^3}\sup_x|g(x,u)|^2 du \leq C \sum_{|\gamma|\leq 2} \|\partial^\gamma g\|^2,\\
   \sup_x \int_{\mathbb{R}^3}\nu(u)|g(x,u)|^2 du\leq C \sum_{|\gamma|\leq 2} \|\partial^\gamma g\|_\nu^2.
\end{gather*}
By using the above embedding, we take the $L^\infty(\mathbb{T}^3)$ norm of $g_1$ or $g_2$ terms on the right-hand side of \eqref{G-ineqn}, specifically selecting the term with the lower total derivative order, i.e.,  $|\beta_i|+|\gamma_i| \leq N/2$ for $i=1,2$. 
Integrating   \eqref{G-ineqn} over $x\in \mathbb{T}^3$ yields \eqref{Gamma-derivative}.
\end{proof}

\begin{lemma}\label{LGamma}
We have
\begin{align}
\left|\left( \Gamma(f,g), h \right)\right|\leq  \sup_{x,v} |\mu^{-1/4}h|\|f\| \|g\|.
\end{align}
\end{lemma}
\begin{proof}
It is easy to know that
\begin{align*}
 \left|\langle \Gamma_{\rm{gain}}(f,g),h \rangle\right|
\leq\,& \int\!\!\int_{\mathbb{R}^3\times \mathbb{R}^3}|u-v| \mu^{1/2}(u)f(u')g(v')h(v)dudv\\
\leq\,& \sup_v |\mu^{-1/4}h| \int\!\!\int_{\mathbb{R}^3\times \mathbb{R}^3}|u-v| \mu^{1/4}(u)\mu^{1/4}(v)f(u)g(v)dudv\\
\leq\,& \sup_v |\mu^{-1/4}h| \left(\int\!\!\int_{\mathbb{R}^3\times \mathbb{R}^3}|u-v| \mu^{1/4}(u)\mu^{1/4}(v)f^2(u)dudv\right)^{1/2}\\
&\qquad \qquad\quad\times\left(\int\!\!\int_{\mathbb{R}^3\times \mathbb{R}^3}|u-v| \mu^{1/4}(u)\mu^{1/4}(v)g^2(v)dudv\right)^{1/2}\\
\leq\,& \sup_v |\mu^{-1/4}h| \left(\int_{\mathbb{R}^3}\langle u \rangle \mu^{1/4}(u)f^2(u)du\right)^{1/2}\left(\int_{\mathbb{R}^3}\langle v \rangle \mu^{1/4}(v)g^2(v)dv\right)^{1/2}\\
\leq\,& \sup_v |\mu^{-1/4}h| |f|_0 |g|_0.
\end{align*}
Integrating the above inequality over $\mathbb{T}^3$ yields
$\left|\left( \Gamma_{\rm{gain}}(f,g), h \right)\right|\leq  \sup_{x,v} |\mu^{-1/4}h|\|f\| \|g\|$.
Similarly, $\Gamma_{\rm{loss}}(\cdot,\cdot)$ also satisfies this estimate. The proof is complete.
\end{proof}
%\textcolor[rgb]{1.00,0.00,0.00}{Remark: I  changed  " This completes the proof of the lemma. " to " Hence the proof is completed.  "}

\section{$L^2$ Estimates of the Poisson-Poincar\'{e} Equation \eqref{p-p-eqn} in $\mathbb{T}^3$}\label{ellip-section}
In this section, we investigate the well-posedness, regularity and quantitative estimates of the Poisson-Poincar\'{e} equation:
\begin{align}\label{PPE}
\Delta U=e^U-\rho\quad \rm{in}\,\,\,\mathbb{T}^3.
\end{align}
For the well-posedness of \eqref{PPE}, we have
\begin{theorem}\label{solution-PPE}
Suppose that $\rho\in L^2(\mathbb{T}^3)$. Then, there exists a unique solution $U\in H^2(\mathbb{T}^3)$ to \eqref{PPE}. Moreover, denoting $\sup_{t\geq 0}\|\rho-1\|_{L^2(\mathbb{T}^3)}=M$, it holds that
$\|U\|_{H^2(\mathbb{T}^3)}\leq CM+C e^{2CM}+C$, where $C>0$ is a constant independent of $U$ and $\rho$.
\end{theorem}

\begin{proof}
We split $U=\overline{U}+\widehat{U}$ such that \eqref{PPE} is decomposed into two equations:
\begin{equation}\label{PPE1}
\left\{ \begin{aligned}
&\Delta \overline{U}=1-\rho \quad \rm{in} \,\,\mathbb{T}^3,\\
&(\overline{U})_{\mathbb{T}^3}
=\frac{1}{|\mathbb{T}^3|}\int_{\mathbb{T}^3}\overline{U}dx=0,
\end{aligned}
\right.
\end{equation}
and
\begin{align}\label{PPE2}
 \Delta \widehat{U}=e^{\overline{U}+\widehat{U}}-1 \quad \rm{in} \,\,\mathbb{T}^3.
\end{align}

First, for the equation \eqref{PPE1},
when $\rho-1\in L^2(\mathbb{T}^3)$, there exists a unique solution $\overline{U}\in H^2(\mathbb{T}^3)$ such that
\begin{align}\label{barU}
  \|\overline{U}\|_{H^2(\mathbb{T}^3)}\leq C \|\rho-1\|_{L^2(\mathbb{T}^3)},
\end{align}
where $C>0$ is a constant independent of $\overline{U}$ and $\rho$.

Second, we consider the equation \eqref{PPE2}
with $\overline{U}\in H^2(\mathbb{T}^3)$.
In light of Proposition 3.5 of \cite{Iacobelli2021b}, where the direct method of calculus of variations is applied, 
there exists a unique weak solution $\overline{U}\in H^1(\mathbb{T}^3)$ to \eqref{PPE2}.

To improve the regularity of $\widehat{U}$, we now deduce that
$e^{\widehat{U}}\in L^2(\mathbb{T}^3)$.
Actually, for $\widehat{U}\in H^1(\mathbb{T}^3)$, consider the truncated function $\widehat{U}_k:=\min\{\widehat{U}, k\}$, $k\in \mathbb{N}$. It holds that
$e^{\widehat{U}_k}\in L^\infty(\mathbb{T}^3)$, $\nabla \widehat{U}\in L^2(\mathbb{T}^3)$ and $\nabla e^{\widehat{U}_k}=e^{\widehat{U}_k} \nabla \widehat{U}\chi_{\{\widehat{U}<k\}}\in L^2(\mathbb{T}^3)$.
Thus $e^{\widehat{U}_k}\in L^\infty \cap H^1(\mathbb{T}^3)$. Using $e^{\widehat{U}_k}$ as a test function to the equation \eqref{PPE2}, we have
\begin{align*}
  0= \int_{\mathbb{T}^3}\left(\nabla \widehat{U}\cdot \nabla e^{\widehat{U}_k} +e^{\overline{U}+\widehat{U}}e^{\widehat{U}_k}-e^{\widehat{U}_k} \right)dx
  \geq   e^{-\|\overline{U}\|_{L^\infty}}\int_{\mathbb{T}^3} e^{\widehat{U}+\widehat{U}_k} dx-\int_{\mathbb{T}^3} e^{\widehat{U}_k}dx.
\end{align*}
By using the monotone convergence theorem, we take $k\rightarrow \infty$ to get
\begin{align*}
   \int_{\mathbb{T}^3} e^{\widehat{U}}dx\geq e^{-\|\overline{U}\|_{L^\infty}}\int_{\mathbb{T}^3} e^{2\widehat{U}} dx,
\end{align*}
which implies that
\begin{align*}
\|e^{\widehat{U}}\|_{L^2(\mathbb{T}^3)}\leq e^{\frac{1}{2}\|\overline{U}\|_{L^\infty}}
\|e^{\widehat{U}}\|^{\frac{1}{2}}_{L^1(\mathbb{T}^3)}.
\end{align*}
Meanwhile, we deduce from \eqref{PPE2}  that
 $$0=\int_{\mathbb{T}^3} \Delta \widehat{U}dx=\int_{\mathbb{T}^3}e^U-1 dx,$$
thus, $$\|e^{\widehat{U}}\|_{L^1(\mathbb{T}^3)} = \int_{\mathbb{T}^3}e^{\widehat{U}}dx\leq e^{\|\overline{U}\|_{L^\infty}}|\mathbb{T}^3|=e^{\|\overline{U}\|_{L^\infty}}$$
and 
\begin{align*}
\|e^{\widehat{U}}\|_{L^2(\mathbb{T}^3)}\leq e^{\frac{1}{2}\|\overline{U}\|_{L^\infty}}
\|e^{\widehat{U}}\|^{\frac{1}{2}}_{L^1(\mathbb{T}^3)}\leq e^{\|\overline{U}\|_{L^\infty}}.
\end{align*}
For the equation \eqref{PPE2}, there exists a constant $C>0$ independent of $\widehat{U}$ and $\overline{U}$ such that
\begin{align*}
  \|\widehat{U}\|_{H^2(\mathbb{T}^3)}\leq &\,C \|e^{\overline{U}+\widehat{U}}-1\|_{L^2(\mathbb{T}^3)}\\
  \leq &\, C \|e^{\widehat{U}}\|_{L^2(\mathbb{T}^3)} e^{\|\overline{U}\|_{L^\infty}}+C\\
  \leq &\, C e^{2\|\overline{U}\|_{L^\infty}}+C.
\end{align*}
Denote $\|\rho-1\|_{L^2(\mathbb{T}^3)}$ by $M$. By \eqref{barU}, we have $$\|\overline{U}\|_{L^\infty(\mathbb{T}^3)}\leq C\|\overline{U}\|_{H^2(\mathbb{T}^3)}\leq C M,$$
where $C>0$ is independent of $\overline{U}$ and $M$.
Hence we get
\begin{align*}
 \|U\|_{H^2(\mathbb{T}^3)}\leq \, \|\overline{U}\|_{H^2(\mathbb{T}^3)}+\|\widehat{U}\|_{H^2(\mathbb{T}^3)}
 \leq \, CM +C e^{2CM}+C,
\end{align*}
where $C>0$ is independent of $U$ and $\rho$. This completes the proof.
\end{proof}
\begin{remark}\label{remark3.2}
In Theorem \ref{solution-PPE}, if we further assume $\sup_{t\geq 0}\|\rho-1\|_{L^2(\mathbb{T}^3)}=M\ll 1$, then there exists a uniform upper bound $C>0$ independent of the solution $U$, such that $\sup_{t\geq 0}\|U\|_{L^\infty(\mathbb{T}^3)}\leq C$. It implies that $\sup_{t\geq 0}\|U\|_{L^\infty(\mathbb{T}^3)}=O(1)$, which is crucial for the quantitative estimates of $\|U\|_{H^2(\mathbb{T}^3)}$ as below.
\end{remark}

As a quantitative estimate, we can control $\|U\|_{H^2(\mathbb{T}^3)}$ by $\|\rho-1\|_{L^2(\mathbb{T}^3)}$ as the following
\begin{proposition}\label{UH2}
Suppose that $U\in H^2(\mathbb{T}^3)$ solves \eqref{PPE} with $\rho \in L^2(\mathbb{T}^3)$. Then it holds
\begin{align*}
  \|U\|_{H^2(\mathbb{T}^3)}\leq C \|\rho-1\|_{L^2(\mathbb{T}^3)}
\end{align*}
for some constant $C=C(\|U\|_{L^\infty(\mathbb{T}^3)})>0$.
Moreover, when $\sup_{t\geq 0}\|\rho-1\|_{L^2(\mathbb{T}^3)}$ is sufficiently small, the constant $C$ is independent of $\|U\|_{L^\infty(\mathbb{T}^3)}$.
\end{proposition}
\begin{proof}
Note that the function $U\in H^2(\mathbb{T}^3)\hookrightarrow C(\mathbb{T}^3)$. For each $x\in \mathbb{T}^3$, by the mean value theorem, there exists a $U_\xi(x)\in [-U(x),U(x)]$ such that $e^{U(x)}-1=U(x)e^{U_\xi(x)}$. It implies that, there exists a function $U_\xi \in L^\infty(\mathbb{T}^3)$ such that $e^U-1=Ue^{U_\xi}$ in $\mathbb{T}^3$, where $e^{-\|U\|_{L^\infty(\mathbb{T}^3)}}\leq e^{U_\xi}\leq e^{\|U\|_{L^\infty(\mathbb{T}^3)}}$.

The equation \eqref{PPE} can thus be rewritten as
\begin{align}\label{U-xi}
 \Delta U=Ue^{U_\xi}+1-\rho.
\end{align}
Taking $L^2$ inner product of the above equation with $U$ yields
\begin{align*}
 \int_{\mathbb{T}^3}|\nabla U|^2+|U|^2 e^{U_\xi}dx=\int_{\mathbb{T}^3}U(\rho-1)dx,
\end{align*}
which implies that $\|U\|_{H^1(\mathbb{T}^3)}\leq e^{\|U\|_{L^\infty(\mathbb{T}^3)}}\|\rho-1\|_{L^2(\mathbb{T}^3)}$. Note that
\begin{align*}
  \|D^2 U\|_{L^2(\mathbb{T}^3)}=\, & \|Ue^{U_\xi}+1-\rho\|_{L^2(\mathbb{T}^3)}\\
  \leq\, & \|U\|_{L^2(\mathbb{T}^3)}e^{\|U\|_{L^\infty(\mathbb{T}^3)}}+\|\rho-1\|_{L^2(\mathbb{T}^3)}\\
  \leq\, & (e^{2\|U\|_{L^\infty(\mathbb{T}^3)}}+1)\|\rho-1\|_{L^2(\mathbb{T}^3)}.
\end{align*}
Hence we get
\begin{align}\label{U-H2-estimate}
\|U\|_{H^2(\mathbb{T}^3)}\leq 3 e^{2\|U\|_{L^\infty(\mathbb{T}^3)}}\|\rho-1\|_{L^2(\mathbb{T}^3)},
\end{align}
which implies that
\begin{align}\label{U-Linfty-estimate}
  \|U\|_{L^\infty(\mathbb{T}^3)}\leq Ce^{2\|U\|_{L^\infty(\mathbb{T}^3)}}\|\rho-1\|_{L^2(\mathbb{T}^3)}.
\end{align}
By Remark \ref{remark3.2}, when $\sup_{t\geq 0}\|\rho-1\|_{L^2(\mathbb{T}^3)}\ll 1$, we have $\sup_{t\geq 0}\|U\|_{L^\infty(\mathbb{T}^3)}=O(1)$. 
Furthermore, 
by the estimate \eqref{U-Linfty-estimate}, we deduce that $\sup_{t\geq 0}\|U\|_{L^\infty(\mathbb{T}^3)}=o(1)$. Hence, by the estimate \eqref{U-H2-estimate}, we have
\begin{align*}
  \|U\|_{H^2(\mathbb{T}^3)}\leq 3 e^{o(1)}\|\rho-1\|_{L^2(\mathbb{T}^3)}\leq 4 \|\rho-1\|_{L^2(\mathbb{T}^3)}.
\end{align*} 
The proof is complete.
\end{proof}
\begin{corollary}\label{U1-U2}
Suppose that $U_i\in H^2(\mathbb{T}^3)$ solves \eqref{PPE} with $\rho_i \in L^2(\mathbb{T}^3)$ for $i=1,2$.
Then,
\begin{align*}
  \|U_1-U_2\|_{H^2(\mathbb{T}^3)}\leq C \|\rho_1-\rho_2\|_{L^2(\mathbb{T}^3)}
\end{align*}
  holds for some constant $C=C(\max\limits_{i=1,2}\|U_i\|_{L^\infty(\mathbb{T}^3)})>0$.
  Moreover, when $\max\limits_{i=1,2}\sup_{t\geq 0}\|\rho_i-1\|_{L^2(\mathbb{T}^3)}$ is sufficiently small, the constant $C$ is independent of $\max\limits_{i=1,2}\|U_i\|_{L^\infty(\mathbb{T}^3)}$.
\end{corollary}
\begin{proof}
Take the difference for $U_1$ and $U_2$ to get
\begin{align*}
  \Delta (U_1-U_2)= e^{U_1}-e^{U_2}-(\rho_1-\rho_2)
  = (U_1-U_2)e^{\overline{U}_\xi}-(\rho_1-\rho_2),
\end{align*}
where the value of $\overline{U}_\xi(x)$ lies between $U_1(x)$ and $U_2(x)$ for every $x\in \mathbb{T}^3$ and hence $|\overline{U}_\xi|\leq \max\limits_{i=1,2}\|U_i\|_{L^\infty(\mathbb{T}^3)}$. Through the similar arguments to that  in Proposition \ref{UH2}, we obtain that
\begin{align}\label{u1u2-h2}
\|U_1-U_2\|_{H^2(\mathbb{T}^3)}\leq 3 e^{2 \max\limits_{i=1,2}\|U_i\|_{L^\infty(\mathbb{T}^3)}}\|\rho_1-\rho_2\|_{L^2(\mathbb{T}^3)}.
\end{align}
Furthermore, by Proposition \ref{UH2}, if $\max\limits_{i=1,2}\sup_{t\geq 0}\|\rho_i-1\|_{L^2(\mathbb{T}^3)}$ is sufficiently small, then 
$$\max\limits_{i=1,2}\sup_{t\geq 0}\|U_i\|_{L^\infty(\mathbb{T}^3)}=o(1).$$
By the estimate \eqref{u1u2-h2}, we have
 \begin{align*}
\|U_1-U_2\|_{H^2(\mathbb{T}^3)}\leq \,3 e^{o(1)}\|\rho_1-\rho_2\|_{L^2(\mathbb{T}^3)}
\leq  \,4 \|\rho_1-\rho_2\|_{L^2(\mathbb{T}^3)}.
\end{align*}
The proof is complete.
\end{proof}

Our next intention is to iterate the argument above to deduce that the solution $U$ lies in higher order Sobolev spaces provided that $\rho$ has sufficient regularities.
%lies in sufficiently good spaces.
\begin{proposition}\label{elliprop}
Suppose that $U\in H^2(\mathbb{T}^3)$ solves \eqref{PPE} with $\rho \in L^2(\mathbb{T}^3)$.
Let $m$ be a positive integer and the multi-index $\gamma=[\gamma_0,\gamma_1,\gamma_2,\gamma_3]$ with $1\leq |\gamma|\leq m$. Assume that
$\sup_{t\geq 0}\|\partial^{\gamma} \rho\|_{L^2(\mathbb{T}^3)}<\infty$ and
\begin{align}\label{rhoL2}
  \sup_{t\geq 0}\sum_{1\leq |\gamma'|\leq m-1}\|\partial^{\gamma'} \rho\|_{L^2(\mathbb{T}^3)}+\sup_{t\geq 0}\|\rho-1\|_{L^2(\mathbb{T}^3)}\ll 1.
\end{align}
Then for any $t\geq 0$ and $1\leq |\gamma|\leq m$, we have
\begin{align}\label{ellipEstimate}
  \|\partial^\gamma U(t,x)\|_{H^2(\mathbb{T}^3)}\leq C\sum_{1\leq |\gamma'|\leq m}\|\partial^{\gamma'} \rho(t,x)\|_{L^2(\mathbb{T}^3)},
\end{align}
where the constant $C$ is independent of $U$ and $\rho$.
\end{proposition}
\begin{proof}
Suppose $U\in H^2(\mathbb{T}^3)$ satisfies $\Delta U=e^U-\rho$ with $\rho\in L^2(\mathbb{T}^3)$, we consider a mapping 
\begin{align*}
  \mathcal{F}:\,\ \ &H^2(\mathbb{T}^3)\times L^2(\mathbb{T}^3)\rightarrow L^2(\mathbb{T}^3)\\
  & (v,w)\mapsto \Delta v-e^v+w.
\end{align*}
Note that $\mathcal{F}(U, \rho)=0$ and the Fr\'{e}chet derivative $\frac{\partial\mathcal{F}}{\partial v}(U,\rho)=\Delta-e^U\in \mathcal{L}(H^2(\mathbb{T}^3);L^2(\mathbb{T}^3))$ is reversible by the Lax-Milgram theorem. Then, we use the implicit function theorem 
to get the Fr\'{e}chet derivative $$U'(\rho)=-\left(\frac{\partial\mathcal{F}}{\partial v}(U,\rho)\right)^{-1}\frac{\partial\mathcal{F}}{\partial w}(U,\rho)=-(\Delta -e^U)^{-1}.$$

For $|\gamma|=1$, by the chain rules, we deduce that $\partial^\gamma U=U'(\rho)\partial^\gamma \rho=-(\Delta -e^U)^{-1}\partial^\gamma \rho$, which is rewritten as
\begin{align*}
 \Delta (\partial^\gamma U)-(\partial^\gamma U) e^U+\partial^\gamma \rho= 0.
\end{align*}
Assuming $\sup_{t\geq 0}\|\rho-1\|_{L^2(\mathbb{T}^3)} \ll 1$ and  $\sup_{t\geq 0}\|\partial^\gamma \rho(t,x)\|_{L^2(\mathbb{T}^3)} <\infty$, 
through the similar procedure to that in Proposition \ref{UH2}, we deduce that
\begin{align*}
 \|\partial^\gamma U\|_{H^2(\mathbb{T}^3)}\leq 3e^{2\|U\|_{L^\infty(\mathbb{T}^3)}}\|\partial^\gamma \rho\|_{L^2(\mathbb{T}^3)}
 \leq 4\|\partial^\gamma \rho\|_{L^2(\mathbb{T}^3)}, \quad \forall \,t\geq 0, \,\,\,|\gamma|=1.
\end{align*}
This verifies the assertion \eqref{ellipEstimate} for $|\gamma|=1$.

For $|\gamma|=2,\ldots,m$, we assume the assertion \eqref{ellipEstimate} holds. We now consider the case of $|\gamma|=m+1$.

From the equation $\Delta U=e^U-\rho$, we have
\begin{align*}
  \Delta \partial^\gamma U-e^U \partial^\gamma U= [\partial^\gamma (e^U)- e^U \partial^\gamma U]-\partial^\gamma \rho
=: \mathcal{S}-\partial^\gamma \rho,
\end{align*}
where $|\gamma|=m+1$, and
\begin{align*}
  \mathcal{S}:=e^U \sum^{m+1}_2 \sum_{\substack{p_1+\cdots + p_k=|\gamma| \\ |p_j|\geq 1}}\frac{|\gamma|!}{p_1!\cdots p_k!}\partial^{p_1}U\cdots \partial^{p_k}U.
\end{align*}
When $$\sup_{t\geq 0}\sum_{1\leq |\gamma'|\leq m}\|\partial^{\gamma'} \rho\|_{L^2(\mathbb{T}^3)}+\sup_{t\geq 0}\|\rho-1\|_{L^2(\mathbb{T}^3)}\ll 1,$$ we have
\begin{align*}
\|\mathcal{S}\|_{L^2(\mathbb{T}^3)}\leq &\,C e^{\|U\|_{L^\infty(\mathbb{T}^3)}}\|\partial^{p_1}U\|_{L^2(\mathbb{T}^3)}
\|\partial^{p_2}U\|_{L^\infty(\mathbb{T}^3)}\cdots \|\partial^{p_k}U\|_{L^\infty(\mathbb{T}^3)}\\
\leq & \,C e^{4 \|\rho-1\|_{L^2(\mathbb{T}^3)}}\sum_{1\leq |\gamma'|\leq m}\|\partial^{\gamma'} \rho\|_{L^2(\mathbb{T}^3)}\\
\leq & \,C \sum_{1\leq |\gamma'|\leq m}\|\partial^{\gamma'} \rho\|_{L^2(\mathbb{T}^3)}.
\end{align*}
Hence, for $|\gamma|=m+1$, we obtain that
\begin{align*}
\|\partial^\gamma U\|_{H^2(\mathbb{T}^3)}\leq & \, 3e^{2\|U\|_{L^\infty(\mathbb{T}^3)}}\|\mathcal{S}-\partial^\gamma \rho\|_{L^2(\mathbb{T}^3)}\\
\leq &\, 4 (\|\mathcal{S}\|_{L^2(\mathbb{T}^3)}+\|\partial^\gamma \rho\|_{L^2(\mathbb{T}^3)})\\
\leq & \,C \sum_{1\leq |\gamma'|\leq m}\|\partial^{\gamma'} \rho\|_{L^2(\mathbb{T}^3)}+C \|\partial^\gamma \rho\|_{L^2(\mathbb{T}^3)}.
\end{align*}
which indicates that \eqref{ellipEstimate} holds for $|\gamma|=m+1$. Inductively, we complete the proof.
\end{proof}
\begin{remark}\label{remark3.6}
Let $\rho=\int_{\mathbb{R}^3} F dv=\int_{\mathbb{R}^3}\sqrt{\mu}fdv+1$ as in \eqref{poisson}.
Combining the results in Proposition \ref{UH2} and Proposition \ref{elliprop}, we obtain that, if $\sup_{t\geq 0}\sum_{|\gamma'|\leq |\gamma|}\|\partial^{\gamma'}f\|_{L^2(\mathbb{T}^3\times \mathbb{R}^3)}\ll 1$, then there exists a constant $C>0$ independent of $U$ and $f$ such that 
\begin{align}
  \|\partial^\gamma U(t,x)\|_{H^2(\mathbb{T}^3)}\leq C\sum_{|\gamma'|\leq |\gamma|}\|\partial^{\gamma'} f(t,x,v)\|_{L^2(\mathbb{T}^3\times \mathbb{R}^3)}, \quad \forall \,t\geq 0.
\end{align}

\end{remark}

\section{Local-in-time well-posedness to the     system \eqref{perturb eqn}--\eqref{poisson} }\label{local-section}
In this section, we construct a unique local-in-time solution to the ionic  Vlasov-Poisson-Boltzmann   system \eqref{perturb eqn}--\eqref{poisson} with a uniform energy estimate as follows.
\begin{theorem}\label{LocalTheorem}
There exist $M_0>0$ and $T_0>0$ such that if $T_0\leq M_0/2$ and $\mathcal{E}(0)\leq M_0/2$,
then the ionic  Vlasov-Poisson-Boltzmann  system \eqref{perturb eqn}--\eqref{poisson}
 has a unique solution $f(t,x,v)$  in $[0, T_0]\times \mathbb{T}^3 \times \mathbb{R}^3$ satisfying
\begin{align*}
 \sup_{0\leq t\leq T_0}\mathcal{E}(t)\leq M_0,
\end{align*}
and $\mathcal{E}(t)$ is continuous over $[0, T_0]$. If $F_0(x,v)=\mu(v)+\sqrt{\mu(v)}f_0(x,v)\geq 0$, then
$F(t,x,v)=\mu(v)+\sqrt{\mu(v)}f(t,x,v)\geq 0$. Furthermore, the conservation laws \eqref{mass}--\eqref{energy} are valid for all $0<t\leq T_0$ if they hold initially.
\end{theorem}
\begin{proof}
Consider the iterating sequence $\{F^n\}_{n\in \mathbb{N}}$ which solves
\begin{align}\label{iterateF}
 \{\partial_t +v\cdot \nabla_x+E^n \cdot \nabla_v +R(F^n)\}F^{n+1}=&\,Q_{\rm{gain}}(F^n,F^n)\nonumber\\
 =&\,\int\!\!\int_{\mathbb{R}^3\times \mathbb{S}^2}|(v-u)\cdot \omega|F^n(v')F^n(u')dud\omega
\end{align}
with the initial datum $F^{n+1}(0,x,v)=F_0(x,v)$, $n\in \mathbb{N}$. Here,
\begin{align*}
  R(F^n):= \int\!\!\int_{\mathbb{R}^3\times \mathbb{S}^2}|(v-u)\cdot \omega|F^n(u)dud\omega.
\end{align*}
We take $F^0(t,x,v)\equiv F_0(x,v)$. Since $F^{n}=\mu+\sqrt{\mu}f^{n}$ for $n\in \mathbb{N}$, we have $f^0(t,x,v)\equiv f_0(x,v)$ and $f^{n+1}(0,x,v)=f_0(x,v)$.
Equivalently to \eqref{iterateF}, $f^{n+1}$ satisfies
\begin{align}\label{iteratef}
 \left\{\partial_t +v\cdot \nabla_x+E^n \cdot \nabla_v +\nu(v)-\frac{v}{2}\cdot E^n\right\}f^{n+1}-E^n \cdot v\sqrt{\mu}=\mathcal{K}f^n+\Gamma_{\rm{gain}}(f^n,f^n)-\Gamma_{\rm{loss}}(f^n,f^{n+1}).
\end{align}
We now prove that, there exist $M_0>0$ and $T_0>0$ such that if $T_0\leq M_0/2$, $\mathcal{E}(0)\leq M_0/2$ and $\sup_{0\leq t\leq T_0}\mathcal{E}_n (t)\leq M_0$, then
$\sup_{0\leq t\leq T_0}\mathcal{E}_{n+1} (t)\leq M_0$. Then, we can establish the local-in-time existence in Theorem \ref{LocalTheorem}, since $\sup_{0\leq t\leq T_0}\mathcal{E}_0 (t)=\mathcal{E}(0)\leq M_0/2$.

Taking $\partial_\beta^\gamma$ derivatives of \eqref{iteratef}, we have
\begin{align}\label{iterate-df}
 & \left\{\partial_t +v\cdot \nabla_x+E^n \cdot \nabla_v -\frac{v}{2}\cdot E^n\right\}\partial_\beta^\gamma f^{n+1}+\partial_\beta [\nu \partial^\gamma f^{n+1}]\nonumber\\
&\quad\, =\partial^\gamma E^n \cdot \partial_\beta(v\sqrt{\mu})-\sum_{\beta_1\neq 0}C_\beta^{\beta_1}\partial_{\beta_1}v\cdot \nabla_x \partial_{\beta-\beta_1}^\gamma f^{n+1}+ \partial_\beta [\mathcal{K} \partial^\gamma f^n]\nonumber\\
 &\qquad \,+ \sum_{\gamma_1,\beta_1\neq 0} C_\beta^{\beta_1}C_\gamma^{\gamma_1}\partial^{\gamma_1}E^n\cdot \partial_{\beta_1}\left(\frac{v}{2}\right)\partial_{\beta-\beta_1}^{\gamma-\gamma_1}f^{n+1} -\sum_{\gamma_1\neq 0} C_\gamma^{\gamma_1}\partial^{\gamma_1}E^n\cdot \nabla_v \partial_{\beta}^{\gamma-\gamma_1}f^{n+1} \nonumber\\
 &\qquad \,+\sum C_\gamma^{\gamma_1} \partial_\beta \left\{\Gamma_{\rm{gain}}(\partial^{\gamma_1}f^n,\partial^{\gamma-\gamma_1}f^n)-\Gamma_{\rm{loss}}
 (\partial^{\gamma_1}f^n,\partial^{\gamma-\gamma_1}f^{n+1})\right\}.
\end{align}

Now we take $L^2$ inner product of \eqref{iterate-df} with $\partial_\beta^\gamma f^{n+1}$, $|\gamma|+|\beta|\leq N$, in $\mathbb{T}^3\times \mathbb{R}^3$ and deal with the result term by term as follows.

(i). First, we have
\begin{align*}
&  \partial_\beta^\gamma f^{n+1}\left\{\partial_t +v\cdot \nabla_x+E^n \cdot \nabla_v -\frac{v}{2}\cdot E^n\right\}\partial_\beta^\gamma f^{n+1}\\
&=   \left\{\partial_t +v\cdot \nabla_x+E^n \cdot \nabla_v -v\cdot E^n\right\}\big(\frac{1}{2}| \partial_\beta^\gamma f^{n+1}|^2\big).
\end{align*}
Taking $g=\partial^\gamma f^{n+1}$ in \eqref{nu-coercive} of Lemma \ref{K}, we get
\begin{align*}
 & \int\!\!\int_{\mathbb{T}^3\times \mathbb{R}^3}\partial_\beta^\gamma f^{n+1}\partial_\beta [\nu \partial^\gamma f^{n+1}]dxdv\\
 &\,\quad   \geq\|\partial_\beta^\gamma f^{n+1}\|^2_\nu-\eta \sum_{|\beta_1|\leq |\beta|}\|\partial_{\beta_1}^\gamma f^{n+1}\|^2-C_\eta \|\partial^\gamma f^{n+1}\|^2
\end{align*}
for some small $\eta>0$.
Then the contribution of the left-hand side of \eqref{iterate-df} is
\begin{align*}
  \frac{1}{2}\frac{d}{dt}\|\partial_\beta^\gamma f^{n+1}(t)\|^2+\int\!\!\int_{\mathbb{T}^3\times \mathbb{R}^3}\big(\nu-\frac{1}{2}v\cdot E^n\big) | \partial_\beta^\gamma f^{n+1}|^2 dxdv-\eta \sum_{|\beta_1|\leq |\beta|}\|\partial_{\beta_1}^\gamma f^{n+1}\|^2-C_\eta \|\partial^\gamma f^{n+1}\|^2.
\end{align*}
For $M_0\ll 1$, we have $\|E^n\|_\infty\leq M_0\ll 1$. Since $\nu(v)\sim 1+|v|$, then $\nu(v)-\frac{1}{2}v\cdot E^n\geq \frac{1}{2}\nu(v)$. Hence the contribution of the left-hand side of \eqref{iterate-df} is reduced by
\begin{align*}
  \frac{1}{2}\frac{d}{dt}\|\partial_\beta^\gamma f^{n+1}(t)\|^2+\frac{1}{2}\|\partial_\beta^\gamma f^{n+1}(t)\|_\nu^2-\eta \sum_{|\beta_1|\leq |\beta|}\|\partial_{\beta_1}^\gamma f^{n+1}(t)\|^2-C_\eta \|\partial^\gamma f^{n+1}(t)\|^2.
\end{align*}

\medskip
(ii). For the right-hand side of \eqref{iterate-df}, we estimate them one by one.

(iia). The first term contributes
\begin{align*}
  \int\!\!\int_{\mathbb{T}^3\times \mathbb{R}^3}\partial^\gamma E^n \cdot \partial_\beta(v\sqrt{\mu})\partial_\beta^\gamma f^{n+1}dxdv\leq C \vvvert E^n\vvvert \vvvert f^{n+1}\vvvert.
\end{align*}

(iib). By using Lemma \ref{K}, the inner product of $\sum_{\beta_1\neq 0}C_\beta^{\beta_1}\partial_{\beta_1}v\cdot \nabla_x \partial_{\beta-\beta_1}^\gamma f^{n+1}$ and $\partial_\beta [\mathcal{K} \partial^\gamma f^n]$ can be controlled by $(\vvvert f^n\vvvert+\vvvert f^{n+1}\vvvert)\vvvert f^{n+1}\vvvert$.

(iic). The second line on the right-hand side is controlled by
\begin{align}\label{En-local}
 &C \sum_{\gamma_1\neq 0,|\beta_1|=1} \int_{\mathbb{T}^3}|\partial^{\gamma_1}E^n|
 \left( \int_{\mathbb{R}^3}  | \partial_{\beta-\beta_1}^{\gamma-\gamma_1}f^{n+1}|^2dv\right)^{\frac{1}{2}}
 \left( \int_{\mathbb{R}^3}  | \partial_{\beta}^{\gamma}f^{n+1}|^2dv\right)^{\frac{1}{2}} dx\nonumber\\
 &\quad + C\sum_{\gamma_1\neq 0}\int_{\mathbb{T}^3}|\partial^{\gamma_1}E^n|
 \left( \int_{\mathbb{R}^3}  |\nabla_v \partial_{\beta}^{\gamma-\gamma_1}f^{n+1}|^2dv\right)^{\frac{1}{2}}
 \left( \int_{\mathbb{R}^3}  | \partial_{\beta}^{\gamma}f^{n+1}|^2dv\right)^{\frac{1}{2}} dx.
\end{align}
Note that
\begin{align*}
  &\sup_x \int_{\mathbb{R}^3}|g(x,u)|^2 du\leq C \sum_{|\gamma|\leq 2} \|\partial^\gamma g\|^2.
\end{align*}
Thus, when $|\gamma_1|\geq 3$, $|\gamma-\gamma_1|+2\leq |\gamma|-1$, the second term of \eqref{En-local} is controlled by
\begin{align*}
  C\|\partial^{\gamma_1}E^n\|\sum_{|\gamma-\gamma_1|+1+2\leq |\gamma|}\|\nabla_v \partial_{\beta}^{\gamma-\gamma_1}f^{n+1}\|
  \vvvert f^{n+1}\vvvert \leq C\|\partial^{\gamma_1}E^n\|\vvvert f^{n+1}\vvvert^2
  \leq C\vvvert E^n\vvvert \vvvert f^{n+1}\vvvert^2.
\end{align*}
When $|\gamma_1|\leq 2$, $|\gamma_1|+2\leq 4\leq N$, the second term of \eqref{En-local} is controlled by
\begin{align*}
  C\|\partial^{\gamma_1}E^n\|_{H^2(\mathbb{T}^3)}\vvvert  f^{n+1}\vvvert^2\leq C\vvvert  E^n\vvvert  \vvvert  f^{n+1}\vvvert^2.
\end{align*}
Adopting the similar argument to the first term of \eqref{En-local}, we deduce that \eqref{En-local} can be controlled by
$$C\vvvert  E^n\vvvert  \vvvert  f^{n+1}\vvvert^2.$$

(iid). By using Lemma \ref{Gamma}, the contribution of the last line on \eqref{iterate-df} is
\begin{align*}
  C(\vvvert  f^{n}\vvvert  \vvvert  f^{n}\vvvert   _\nu\vvvert  f^{n+1}\vvvert_\nu
  +\vvvert  f^{n+1}\vvvert\vvvert  f^{n}\vvvert_\nu\vvvert  f^{n+1}\vvvert_\nu
  +\vvvert  f^{n}\vvvert  \vvvert  f^{n+1}\vvvert^2_\nu).
\end{align*}

Collecting the above estimates up and summing over $|\beta|+|\gamma|\leq N$, we get
\begin{align}\label{localEn}
  \frac{1}{2}\mathcal{E}_{n+1}'(t):=\,&\frac{1}{2}\frac{d}{dt}\vvvert  f^{n+1}\vvvert^2
  +\frac{1}{2}\vvvert  f^{n+1}\vvvert^2_\nu\nonumber\\
  \leq\, & C \vvvert  E^n\vvvert  \left(\vvvert  f^{n+1}\vvvert
  + \vvvert  f^{n+1}\vvvert^2\right)
   +C(\vvvert  f^n\vvvert+\vvvert  f^{n+1}\vvvert)\vvvert  f^{n+1}\vvvert\nonumber\\
  &+ C(\vvvert  f^{n}\vvvert  \vvvert  f^{n}\vvvert_\nu\vvvert  f^{n+1}\vvvert_\nu
  +\vvvert  f^{n+1}\vvvert   \vvvert  f^{n}\vvvert_\nu\vvvert  f^{n+1}\vvvert_\nu
  +\vvvert  f^{n}\vvvert  \vvvert  f^{n+1}\vvvert^2_\nu)\nonumber\\
  \leq\, & C\{\vvvert  f^{n}\vvvert^2+\vvvert  f^{n+1}\vvvert   ^2+\vvvert  f^{n}\vvvert(\vvvert  f^{n}\vvvert   _\nu^2+\vvvert  f^{n+1}\vvvert_\nu^2+\vvvert  f^{n+1}\vvvert^2)\}\nonumber\\
  &\,+C \vvvert  f^{n+1}\vvvert\vvvert  f^{n}\vvvert_\nu \vvvert  f^{n+1}\vvvert_\nu,
\end{align}
where the last inequality relies on the elliptic estimates established in Proposition \ref{elliprop} that $\vvvert  E^{n}\vvvert   \leq C \vvvert  f^{n}\vvvert   $. To investigate the term $\vvvert  f^{n+1}\vvvert\vvvert  f^{n}\vvvert   _\nu \vvvert  f^{n+1}\vvvert   _\nu$, we notice that
\begin{align*}
   \int_0^t \vvvert  f^{n+1}\vvvert   \vvvert  f^{n}\vvvert   _\nu \vvvert  f^{n+1}\vvvert   _\nu ds
  & \leq \sup_{0\leq s\leq t}\vvvert  f^{n+1}\vvvert  (s)\left\{\int_0^t \vvvert  f^{n}\vvvert   _{\nu}^2  ds\right\}^{1/2}
  \left\{\int_0^t \vvvert  f^{n+1}\vvvert   _{\nu}^2  ds\right\}^{1/2}\\
  & \leq M_0^{1/2} \sup_{0\leq s\leq t}\mathcal{E}_{n+1}(s).
\end{align*}
Hence, integrating \eqref{localEn} over $[0,t]$ yields
\begin{align*}
  \mathcal{E}_{n+1}(t)\leq \mathcal{E}_{n+1}(0)+C \Big\{tM_0+t \sup_{0\leq s\leq t}\mathcal{E}_{n+1}(s)+M_0^{3/2}+M_0^{1/2} \sup_{0\leq s\leq t}\mathcal{E}_{n+1}(s)\Big\}.
\end{align*}
With the initial datum $f^{n+1}(0,x,v)=f_0(x,v)$, we inductively deduce that $\partial_\beta^\gamma f^{n+1}(0,x,v)=\partial_\beta^\gamma f_0(x,v)$ over the number of temporal derivatives. Thus $\mathcal{E}_{n+1}(0)\equiv \mathcal{E} (f_0)\leq M_0/2$. For $0<t\leq T_0$, it holds
\begin{align*}
  (1-CT_0-C M_0^{1/2})\sup_{0\leq t\leq T_0}\mathcal{E}_{n+1}(t)
  \leq\, &\mathcal{E}_{n+1}(0)+C \{M_0 T_0+M_0^{3/2}\}\\
  \leq\, & M_0/2+C \{M_0 T_0+M_0^{3/2}\}.
\end{align*}
If $T_0\leq M_0/2\ll 1$, we have $\sup_n\sup_{0\leq t\leq T_0}\mathcal{E}_{n}(t)\leq M_0$. 
Thus the local-in-time existence in Theorem \ref{LocalTheorem} is
proved.
We denote the solution to the ionic  Vlasov-Poisson-Boltzmann  system \eqref{perturb eqn}--\eqref{poisson} by $f$.

Next, we sketch the argument for proving the uniqueness of solution $f$. If there is another solution $g$ such that $\sup_{0\leq t\leq T_0}\mathcal{E}(g)(t)\leq M_0$, we investigate the $L^2$ energy of $f-g$
via the
equation
\begin{align}\label{f-g}
 & \{\partial_t+v\cdot \nabla_x +E_f \cdot \nabla_v +\nu\}(f-g)\nonumber\\
 &\quad     =(E_f-E_g)\cdot v\sqrt{\mu}-(E_f-E_g)\cdot \nabla_v g
  +\frac{v}{2}\cdot E_f (f-g)+\frac{v}{2}\cdot (E_f-E_g)g\nonumber\\
  &  \quad \,\,\, \, \,\,\, + \mathcal{K}(f-g)+\Gamma(f-g,f)+\Gamma(g,f-g)
\end{align}
with $(f-g)|_{t=0}=f(0,x,v)-g(0,x,v)=0$,
where $(E_f,\phi_f)$ and $(E_g,\phi_g)$ satisfy
\begin{equation*}
\left\{ \begin{aligned}
&E_f=-\nabla_x \phi_f,\\
&\Delta \phi_f=e^{\phi_f}-1-\int_{\mathbb{R}^3}\sqrt{\mu}f dv,
\end{aligned}
\right.
\end{equation*}
and
\begin{equation*}
\left\{ \begin{aligned}
&E_g=-\nabla_x \phi_g,\\
&\Delta \phi_g=e^{\phi_g}-1-\int_{\mathbb{R}^3}\sqrt{\mu}g dv,
\end{aligned}
\right.
\end{equation*}
respectively. Recalling Corollary \ref{U1-U2}, we get the crucial estimate
\begin{align*}
  \|\phi_f-\phi_g\|_{H^2(\mathbb{T}^3)}\leq C \|f-g\|_{L^2(\mathbb{T}^3\times \mathbb{R}^3)},
\end{align*}
for the constant $C>0$ independent of $\phi_f$ and $\phi_g$.
Taking the inner product of \eqref{f-g} with $f-g$ and noticing  Remark \ref{gamma-zero}, we deduce that
\begin{align*}
  &\frac{1}{2}\frac{d}{dt}\|f-g\|^2+\frac{1}{2}\|f-g\|^2_{\nu}\\
  \leq\, & \|E_f-E_g\|\|f-g\|+\sqrt{M_0}\|f-g\|^2_{\nu}+\sqrt{M_0}\|E_f-E_g\|\|f-g\|_\nu+C \|f-g\|^2\\
  &+\bigg\{ \sum_{|\gamma|\leq 2}[\|\partial^\gamma f\|+\|\partial^\gamma g\|]\bigg\}\|f-g\|_\nu^2
  + \bigg\{ \sum_{|\gamma|\leq 2}[\|\partial^\gamma f\|_\nu+\|\partial^\gamma g\|_\nu]\bigg\}\|f-g\|\|f-g\|_\nu\\
  \leq\, & C\left(\vvvert f \vvvert_\nu^2+\vvvert  g\vvvert   _\nu^2+1\right)\|f-g\|^2
  +(C\sqrt{M_0}+\frac{3}{4})\|f-g\|_{\nu}^2.
\end{align*}
For $\eta\ll 1$ and $M_0\ll 1$, it follows that
\begin{align*}
  \frac{1}{2}\|f-g\|^2(t)\leq \frac{1}{2}\|f-g\|^2(0)+C \int_0^t \left(\vvvert f \vvvert_\nu^2+\vvvert  g\vvvert   _\nu^2+1\right)\|f-g\|^2 (s)ds,
\end{align*}
which indicates that
$f(t)\equiv g(t)$ by Gronwall's inequality.

The  continuity of $\mathcal{E}(t)$ comes  from \eqref{localEn}, and the positivity of $F=\mu+\sqrt{\mu}f$ comes from a simple induction over $F^n, n\in \mathbb{N}$. This completes the proof of Theorem \ref{LocalTheorem}.
\end{proof}

\section{Coercivity of $\mathcal{L}$}\label{coerci-section}

To extend the local solution of the ionic  Vlasov-Poisson-Boltzmann  system \eqref{perturb eqn}--\eqref{poisson} to be a global
one, a crucial step is to establish the coercivity of the linearized collision operator $\mathcal{L}$.
We have
\begin{theorem}\label{coercivityThm}
Let $f(t,x,v)$ be a classical solution to the system \eqref{perturb eqn}--\eqref{poisson} satisfying
\eqref{mass}--\eqref{energy}. There exist $M_0\ll 1$ and $\delta_0=\delta_0(M_0)>0$ such that if
\begin{align*}
 \sum_{|\gamma|\leq N}\|\partial^\gamma f(t)\|^2\leq M_0,
\end{align*}
then
\begin{align}\label{coercivity}
  \sum_{|\gamma|\leq N}( \mathcal{L}\partial^\gamma f(t), \partial^\gamma f(t))\geq \delta_0 \sum_{|\gamma|\leq N}\|\partial^\gamma f(t)\|_\nu^2.
\end{align}

\end{theorem}

\begin{proof}

Recall the decomposition $f=\mathbf{P}f+\{\mathbf{I}-\mathbf{P}\}f$ in Section \ref{prelimi-section} and expand the macroscopic part of $f$ as $$\mathbf{P}f(t,x,v)=\{a^f(t,x)+b^f(t,x)\cdot v+c^f(t,x)|v|^2\}\sqrt{\mu}.$$
In the sequel, we denote $[a,b,c]:=[a^f,b^f,c^f]$ for notational simplification.
For the time-space derivatives $\partial^\gamma=\partial_t^{\gamma_0}\partial_{x_1}^{\gamma_1}\partial_{x_2}^{\gamma_2}\partial_{x_3}^{\gamma_3}$, we immediately know that $\partial^\gamma \mathbf{P}f=\mathbf{P}\partial^\gamma f$, and
\begin{align*}
 \|\partial^\gamma \mathbf{P}f\|^2+\|\partial^\gamma \{\mathbf{I}-\mathbf{P}\}f\|^2=\|\partial^\gamma f\|^2.
\end{align*}
Moreover, there exists $C>1$ such that
\begin{align*}
  \frac{1}{C}\|\partial^\gamma \mathbf{P}f\|_\nu^2\leq \|\partial^\gamma a\|^2+\|\partial^\gamma b\|^2+\|\partial^\gamma c\|^2\leq C\|\partial^\gamma \mathbf{P}f\|^2.
\end{align*}
To prove \eqref{coercivity} in Theorem \ref{coercivityThm}, since there exists a $\delta>0$ such that
\begin{align*}
  ( \mathcal{L}\partial^\gamma f,\partial^\gamma f)\geq \delta \|\{\mathbf{I}-\mathbf{P}\}\partial^\gamma f\|_\nu^2,
\end{align*}
it suffices to show that
\begin{align*}
  \sum_{|\gamma|\leq N}\|\mathbf{P}\partial^\gamma f(t)\|_\nu\leq C \sum_{|\gamma|\leq N}\|\{\mathbf{I}-\mathbf{P}\}\partial^\gamma f(t)\|_\nu.
\end{align*}
We thus only need to verify that
\begin{align}\label{abc}
 \sum_{|\gamma|\leq N}\{\|\partial^\gamma a\|+\|\partial^\gamma b\|+\|\partial^\gamma c\|\}
 \leq C\sum_{|\gamma|\leq N}\|\{\mathbf{I}-\mathbf{P}\}\partial^\gamma f(t)\| +C\sqrt{M_0}\sum_{|\gamma|\leq N}\|\partial^\gamma f(t)\|
\end{align}
holds for $M_0\ll 1$.

To prove \eqref{abc}, we present several lemmas below.
\begin{lemma}\label{lemma-ac}
Under the same condition as in Theorem \ref{coercivityThm}, we have
\begin{gather*}
  \left|\int_{\mathbb{T}^3}a(t,x)dx\right|
  +\left|\int_{\mathbb{T}^3}c(t,x)dx\right|\leq C \{\|\phi\|^2+\|\nabla_x \phi\|^2\},\\
  \int_{\mathbb{T}^3}b(t,x)dx\equiv 0.
\end{gather*}
\end{lemma}
\begin{proof}
Plugging $\mathbf{P}f=\{a+b\cdot v+c|v|^2\}\sqrt{\mu}$ into the conservation laws \eqref{mass}--\eqref{energy}, we deduce that
\begin{gather*}
   \varrho_0 \int_{\mathbb{T}^3} a dx+\varrho_2 \int_{\mathbb{T}^3} c dx=0,\\
   \varrho_2 \int_{\mathbb{T}^3} b \,dx=0,\\
   \varrho_2 \int_{\mathbb{T}^3} a dx+\varrho_4 \int_{\mathbb{T}^3} c dx+2\int_{\mathbb{T}^3} \phi e^\phi dx+\int_{\mathbb{T}^3} |\nabla \phi|^2dx=0,
\end{gather*}
where $\varrho_i=\int_{\mathbb{R}^3} |v|^i \mu(v)dv, i=1,\dots, 4$. Note that $\varrho_0\varrho_4\neq \varrho_2^2 $ and
 $\varrho_0\varrho_4, \varrho_2^2>0$. First, we have
$\int_{\mathbb{T}^3} b \,dx=0$ and
\begin{align*}
  \left|\int_{\mathbb{T}^3}adx\right|+\left|\int_{\mathbb{T}^3}cdx\right|
  \leq C\int_{\mathbb{T}^3}\phi e^\phi dx+\int_{\mathbb{T}^3} |\nabla \phi|^2 dx.
\end{align*}
By Proposition \ref{UH2}, when $M_0\ll 1$, it holds
\begin{align*}
  \|\phi\|_{L^\infty(\mathbb{T}^3)}\leq C\|\phi\|_{H^2(\mathbb{T}^3)}\leq C \|f\|_{L^2(\mathbb{T}^3\times \mathbb{R}^3)}\leq C \sqrt{M_0}\ll 1.
\end{align*}
Thus, by \eqref{neutural},
\begin{align*}
  \int_{\mathbb{T}^3}\phi e^\phi dx=&\int_{\mathbb{T}^3}\left(\phi e^\phi -e^\phi+1\right)dx=\int_{\mathbb{T}^3}\left\{\phi^2+(\phi-1)[e^\phi-\phi-1] \right\}dx\\
  \leq &\int_{\mathbb{T}^3}\phi^2+C(\|\phi\|_\infty+1)|\phi|^2 dx\leq C \int_{\mathbb{T}^3}\phi^2 dx,
\end{align*}
which implies the desired result.
\end{proof}

Rewrite the perturbed Vlasov-Boltzmann equation \eqref{perturb eqn} as
\begin{align}\label{pert-lh}
  \{\partial_t +v\cdot \nabla_x\}\mathbf{P}f-E\cdot v\sqrt{\mu}=l(\{\mathbf{I}-\mathbf{P}\}f)+h(f),
\end{align}
where
\begin{align*}
  l(\{\mathbf{I}-\mathbf{P}\}f):=&-\{\partial_t +v\cdot \nabla_x+\mathcal{L}\}\{\mathbf{I}-\mathbf{P}\}f,\\
  h(f):=&-E\cdot \nabla_v f+\frac{v}{2}\cdot E f+\Gamma(f,f).
\end{align*}
We plug $\mathbf{P}f=\{a+b\cdot v+c|v|^2\}\sqrt{\mu}$ into the left-hand side of \eqref{pert-lh} (denote $\partial^0=\partial_t$, $\partial^j=\partial_{x_j}$, $\partial^i=\partial_{x_i}$) to get
\begin{align*}
  \bigg[\{\partial^i c\} v_i|v|^2+\{\partial^0 c+\partial^i b_i\}v_i^2+\sum_{i<j}\{\partial^i b_j+\partial^j b_i\}v_iv_j
  +\{\partial^0b_i+\partial^i a-E_i\}v_i+\partial^0 a\bigg]\sqrt{\mu}.
\end{align*}
Under the basis $[1,v_i, v_iv_j,v_i^2,v_i|v|^2]\sqrt{\mu}$ ($1\leq i<j\leq 3$), we define
\begin{align}\label{def-component-l+h}
 l+h:=  (l_c+h_c)\cdot v|v|^2 \sqrt{\mu}+(l_i+h_i)v_i^2 \sqrt{\mu}+\sum_{i< j}(l_{ij}+h_{ij})v_iv_j \sqrt{\mu}+(l_{bi}+h_{bi})v_i\sqrt{\mu}+(l_a+h_a)\sqrt{\mu}.
\end{align}

Comparing with the coefficients on both sides of \eqref{pert-lh}, we deduce that
\begin{align}
  &\nabla_x c=l_c+h_c,\label{lc}\\
  &\partial^0 c+\partial^i b_i=l_i+h_i,\label{li}\\
  &\partial^i b_j+\partial^j b_i=l_{ij}+h_{ij},\,\,i\neq j,\label{lij}\\
  &\partial^0 b_i+\partial^i a-E_i=l_{bi}+h_{bi},\label{lbi}\\
  &\partial^0 a=l_a+h_a.\label{la}
\end{align}
\begin{lemma}\label{lemma-l}
Let $\gamma=[\gamma_0,\gamma_1,\gamma_2,\gamma_3]$. Then for any $1\leq i\neq j\leq 3$, it holds
\begin{align*}
 \sum_{|\gamma|\leq N-1}\|\partial^\gamma l_c\|+\|\partial^\gamma l_i\|+\|\partial^\gamma l_{ij}\|+\|\partial^\gamma l_{bi}\|+\|\partial^\gamma l_a\|\leq C \sum_{|\gamma|\leq N}\|\{\mathbf{I}-\mathbf{P}\}\partial^\gamma f\|.
\end{align*}
\end{lemma}
\begin{proof}
Denote $\{\epsilon_n(v), n=1, \ldots, 13\}$ to be the basis $[1, v_i, v_i^2, v_iv_j, v_i|v|^2]\sqrt{\mu}$. Then for fixed $(t,x)$, $\partial^\gamma l_c$, $\partial^\gamma l_i$, $\partial^\gamma l_{ij}$, $\partial^\gamma l_{bi}$ and $\partial^\gamma l_a$ take the forms of the linear combinations of
\begin{align*}
  \int_{\mathbb{R}^3}\partial^\gamma l(\{\mathbf{I}-\mathbf{P}\}f)(t,x,v)\cdot \epsilon_n (v)dv.
\end{align*}
By H\"{o}lder's inequality and Lemma \ref{K}, it holds
\begin{align*}
 &\left\|\int_{\mathbb{R}^3}\partial^\gamma l(\{\mathbf{I}-\mathbf{P}\}f)(t,x,v)\cdot \epsilon_n (v)dv\right\|^2\\
 =\,&\left\|\int_{\mathbb{R}^3} \{\partial_t +v\cdot \nabla_x +\mathcal{L}\}[\{\mathbf{I}-\mathbf{P}\}\partial^\gamma f](t,x,v)\cdot \epsilon_n (v)dv\right\|^2\\
 \leq\, & C\int_{\mathbb{T}^3\times \mathbb{R}^3}|\epsilon_n(v)|\left(|\{\mathbf{I}-\mathbf{P}\}\partial_t \partial^\gamma f|^2+|v|^2
 |\{\mathbf{I}-\mathbf{P}\}\nabla_x \partial^\gamma f|^2+|\mathcal{L}\{\mathbf{I}-\mathbf{P}\} \partial^\gamma f|^2
 \right)\\
 \leq\, & C\|\{\mathbf{I}-\mathbf{P}\}\partial^0 \partial^\gamma f\|^2+C\|\{\mathbf{I}-\mathbf{P}\}\nabla_x \partial^\gamma f\|^2+\int_{\mathbb{T}^3\times \mathbb{R}^3}|\nu^2\epsilon_n(v)|\left(|\nu^{-1}\mathcal{L}\{\mathbf{I}-\mathbf{P}\} \partial^\gamma f|^2
 \right)\\
 \leq\, & C \left\{ \|\{\mathbf{I}-\mathbf{P}\}\partial^0 \partial^\gamma f\|^2+\|\{\mathbf{I}-\mathbf{P}\}\nabla_x \partial^\gamma f\|^2+\|\{\mathbf{I}-\mathbf{P}\} \partial^\gamma f\|^2\right\}.
\end{align*}
Hence the lemma is proved.
\end{proof}
\begin{lemma}\label{lemma-h}
Let $\gamma=[\gamma_0,\gamma_1,\gamma_2,\gamma_3]$. Suppose that
$\sum_{|\gamma|\leq N}\|\partial^\gamma f(t)\|^2\leq M_0$ holds for some $M_0>0$,
then for any $1\leq i\neq j\leq 3$,
\begin{align*}
 \sum_{|\gamma|\leq N}\|\partial^\gamma h_c\|+\|\partial^\gamma h_i\|+\|\partial^\gamma h_{ij}\|+\|\partial^\gamma h_{bi}\|+\|\partial^\gamma h_a\|\leq C \sqrt{M_0}\sum_{|\gamma|\leq N}\|\partial^\gamma f\|.
\end{align*}

\end{lemma}
\begin{proof}
Recall the basis $\{\epsilon_n(v), n=1, \ldots, 13\}$. Then for fixed $(t,x)$, $\partial^\gamma h_c$, $\partial^\gamma h_i$, $\partial^\gamma h_{ij}$, $\partial^\gamma h_{bi}$ and $\partial^\gamma h_a$ take the forms of the linear combinations of
\begin{align*}
  \int_{\mathbb{R}^3}\partial^\gamma h(f)(t,x,v)\cdot \epsilon_n (v)dv
   = \int_{\mathbb{R}^3}\partial^\gamma \left(-E\cdot \nabla_v f+\frac{v}{2}\cdot E f+\Gamma(f,f)\right)(t,x,v)\cdot \epsilon_n (v)dv.
\end{align*}
Note that
\begin{align*}
  -\int_{\mathbb{R}^3}\partial^\gamma (E\cdot \nabla_v f)\cdot \epsilon_n (v)dv
 =\,&\int_{\mathbb{R}^3} \partial^\gamma (Ef)\cdot \nabla_v \epsilon_n(v)dv\\
 \leq \,& \sum_{\gamma_1}C_\gamma^{\gamma_1}\int_{\mathbb{R}^3}|\partial^{\gamma_1}E||\partial^{\gamma-\gamma_1}f||\nabla_v \epsilon_n (v)|dv\\
 \leq \,& C\sum_{\gamma_1}|\partial^{\gamma_1}E|\left\{ \int_{\mathbb{R}^3} |\mu^{1/8}\partial^{\gamma-\gamma_1}f|^2 dv\right\}^{1/2}.
\end{align*}
Thus
\begin{align*}
   \int_{\mathbb{T}^3}\left|\int_{\mathbb{R}^3}\partial^\gamma \left(-E\cdot \nabla_v f+\frac{v}{2}\cdot E f\right)\cdot \epsilon_n (v)dv\right|^2 dx
  \leq   C\sum_{\gamma_1}\int\!\!\int_{\mathbb{T}^3\times \mathbb{R}^3}|\partial^{\gamma_1}E|^2  |\mu^{1/8}\partial^{\gamma-\gamma_1}f|^2 dv dx.
\end{align*}
If $|\gamma_1|\leq N/2$ $(N\geq 4)$, by $H^2(\mathbb{T}^3)\hookrightarrow L^\infty(\mathbb{T}^3)$ and Proposition \ref{elliprop}, we have
\begin{align*}
  \sup_x |\partial^{\gamma_1}E|\leq C\sum_{|\gamma|\leq N}\|\partial^{\gamma}E\|\leq C \sum_{|\gamma|\leq N}\|\partial^{\gamma}f\|\leq C \sqrt{M_0},
\end{align*}
and
\begin{align*}
  \left(\int\!\!\int_{\mathbb{T}^3\times \mathbb{R}^3}  |\mu^{1/8}\partial^{\gamma-\gamma_1}f|^2 dv dx\right)^{1/2}\leq C \sum_{|\gamma|\leq N}\|\partial^\gamma f\|.
\end{align*}
If $|\gamma_1|\geq N/2$, then  $|\gamma-\gamma_1|\leq N/2$. Thus, by $H^2(\mathbb{T}^3)\hookrightarrow L^\infty(\mathbb{T}^3)$, it holds
\begin{align*}
  \sup_x \int_{\mathbb{R}^3}|\partial^{\gamma-\gamma_1}f|^2 dv\leq \int_{\mathbb{R}^3}\sup_x|\partial^{\gamma-\gamma_1}f|^2 dv
  \leq C \sum_{|\gamma|\leq N}\|\partial^{\gamma}f\|^2\leq C (\sqrt{M_0})^2,
\end{align*}
and
\begin{align*}
  \sup_x |\partial^{\gamma_1}E|\leq C\sum_{|\gamma|\leq N}\|\partial^{\gamma}E\|\leq C \sum_{|\gamma|\leq N}\|\partial^{\gamma}f\|.
\end{align*}
Hence the contribution of $-E\cdot \nabla_v f+\frac{v}{2}\cdot E f$ is $C \sqrt{M_0}\sum_{|\gamma|\leq N}\|\partial^\gamma f\|$.

For the term $\Gamma(f,f)$, by Lemma \ref{LGamma}, we deduce that
\begin{align*}
  \left\|\int_{\mathbb{R}^3} \partial^\gamma \Gamma(f,f)\cdot \epsilon_n (v)dv\right\|
 &\,\leq \sum C_\gamma^{\gamma_1}\left\|\int_{\mathbb{R}^3} \Gamma( \partial^{\gamma_1}f,\partial^{\gamma-\gamma_1}f)\cdot \epsilon_n (v)dv\right\|\\
 &\, \leq C \sqrt{M_0}\sum_{|\gamma|\leq N}\|\partial^{\gamma}f\|.
\end{align*}
Thus, the proof is complete.
\end{proof}

Based on the above facts, we come back to prove \eqref{abc}, i.e.,
\begin{align*}
 \sum_{|\gamma|\leq N}\{\|\partial^\gamma a\|+\|\partial^\gamma b\|+\|\partial^\gamma c\|\}
 \leq C\sum_{|\gamma|\leq N}\|\{\mathbf{I}-\mathbf{P}\}\partial^\gamma f(t)\| +C\sqrt{M_0}\sum_{|\gamma|\leq N}\|\partial^\gamma f(t)\|
\end{align*}
holds for $M_0\ll 1$.
We divide the proof into  several steps as follows.

\smallskip
\emph{Step 1. The estimate of $\nabla \partial^\gamma b$, $|\gamma|\leq N-1$.}
By using the identities \eqref{li} and \eqref{lij}, we deduce that (as in \cite{Guo2003})
\begin{align*}
  \Delta \partial^\gamma b_i
  =\,&\sum_{j\neq i}[\partial^{jj}\partial^\gamma b_i] +\partial^{ii}\partial^\gamma b_i\\
  =\,& \sum_{j\neq i}[-\partial^{ij}\partial^\gamma b_j+\partial^j \partial^\gamma l_{ij}+\partial^j \partial^\gamma h_{ij}]
  +[\partial^i \partial^\gamma l_{i}+\partial^i \partial^\gamma h_{i}-\partial^0\partial^i\partial^\gamma c]\\
  =\,& \sum_{j\neq i} [\partial^0\partial^i\partial^\gamma c-\partial^i \partial^\gamma l_j-\partial^i \partial^\gamma h_j]
  -\partial^0\partial^i\partial^\gamma c\\
  &+\sum_{j\neq i}[\partial^j \partial^\gamma l_{ij}+\partial^j \partial^\gamma h_{ij}]+\partial^i \partial^\gamma l_{i}+\partial^i \partial^\gamma h_{i}\\
  =\,&\partial^0 \partial^i \partial^\gamma c+\sum_{j\neq i}[-\partial^i \partial^\gamma l_j-\partial^i \partial^\gamma h_j+\partial^i \partial^\gamma l_{ij}+\partial^i \partial^\gamma h_{ij}]+\partial^i \partial^\gamma l_i+\partial^i \partial^\gamma h_i\\
  =\,& - \partial^{ii} \partial^\gamma b_i+\sum_{j\neq i}[-\partial^i \partial^\gamma l_j-\partial^i \partial^\gamma h_j+\partial^i \partial^\gamma l_{ij}+\partial^i \partial^\gamma h_{ij}]+2\{\partial^i \partial^\gamma l_i+\partial^i \partial^\gamma h_i\}.
\end{align*}
Multiplying the above equality   by $\partial^\gamma b_i$, we get
\begin{align*}
  \|\nabla \partial^\gamma b_i\|\leq &\,C \sum_{i,j}\{\|\partial^\gamma l_j\|+\|\partial^\gamma h_j\|+\|\partial^\gamma l_{ij}\|+\|\partial^\gamma h_{ij}\|+\|\partial^\gamma l_i\|+\|\partial^\gamma h_i\|\}\\
  \leq & \,C \sum_{|\gamma|\leq N}\|\{\mathbf{I}-\mathbf{P}\}\partial^\gamma f(t)\| +C\sqrt{M_0}\sum_{|\gamma|\leq N}\|\partial^\gamma f(t)\|,
\end{align*}
where we have   used Lemmas \ref{lemma-l} and   \ref{lemma-h}.

\smallskip
\emph{Step 2. The estimate of $\partial^\gamma c$, $|\gamma|\leq N$.}
For $|\gamma'|\leq N-1$, from the identities \eqref{lc} and \eqref{li}, it holds
\begin{align*}
  \|\partial^0 \partial^{\gamma'} c\|+\|\nabla \partial^{\gamma'} c\|
  &\leq \|\partial^i \partial^{\gamma'} b_i\|+\| \partial^{\gamma'} l_i\|+\| \partial^{\gamma'} h_i\|+\| \partial^{\gamma'} l_c\|+\| \partial^{\gamma'} h_c\|.
\end{align*}
In addition, applying Poincar\'{e}'s inequality, Lemma \ref{lemma-ac} and Proposition \ref{UH2}, we get
\begin{align*}
  \|c\|\leq C\bigg\{\|\nabla c\|+\bigg|\int_{\mathbb{T}^3} c dx\bigg|\bigg\}\leq   C \{\|\nabla c\|
  +\|\phi\|^2+\|\nabla_x \phi\|^2\}
  \leq   C \{\|\nabla c\|+\|f\|^2\}.
\end{align*}
Combining the above estimates on $c(t,x)$ and its temporal and spatial derivatives yields
\begin{align*}
 \sum_{|\gamma|\leq N} \|\partial^\gamma c\|\leq  C \sum_{|\gamma|\leq N}\|\{\mathbf{I}-\mathbf{P}\}\partial^\gamma f(t)\| +C\sqrt{M_0}\sum_{|\gamma|\leq N}\|\partial^\gamma f(t)\|.
\end{align*}
Here,  Lemma \ref{lemma-l} and Lemma \ref{lemma-h} have been used.

\smallskip
\emph{Step 3. The estimate of $\partial^\gamma a$, $|\gamma|\leq N$.}
For $|\gamma'|\leq N-1$, from the identity \eqref{la}, we have
\begin{align*}
  \|\partial_t \partial^{\gamma'} a\|\leq &\|\partial^{\gamma'} l_a\|+\|\partial^{\gamma'} h_a\|
  \leq   C \sum_{|\gamma|\leq N}\|\{\mathbf{I}-\mathbf{P}\}\partial^\gamma f(t)\| +C\sqrt{M_0}\sum_{|\gamma|\leq N}\|\partial^\gamma f(t)\|.
\end{align*}
We now focus on purely spatial derivatives of $a(t,x)$.
Let $\gamma'=[0, \gamma_1',\gamma_2',\gamma_3']\neq 0$ ($0<|\gamma'|\leq N-1$). Taking $\partial^i \partial^{\gamma'}$ to the three identities \eqref{lbi}, \eqref{lc} and $E_i=-\partial^i \phi$, and then summing over $i=1,2,3$, we get
\begin{gather}
  -\Delta \partial^{\gamma'}a+\nabla\cdot \partial^{\gamma'} E=\nabla\cdot \partial^0 \partial^{\gamma'} b-\sum_i\partial^i \partial^{\gamma'}(l_{bi}+h_{bi}),\label{lbi'}\\
  \Delta \partial^{\gamma'} c=\nabla\cdot \partial^{\gamma'} (l_c+h_c),\label{lc'}
\end{gather}
and
\begin{align}\label{phi'}
  \Delta \partial^{\gamma'} \phi=- \nabla\cdot \partial^{\gamma'} E.
\end{align}
We take \eqref{lbi'}$\times [\varrho_0 \partial^{\gamma'} a+\varrho_2 \partial^{\gamma'} c]-$\eqref{lc'}$\times [K \partial^{\gamma'} c]-$\eqref{phi'}$\times [\partial^{\gamma'}(e^\phi-1)]$, where the constant $K>0$ is sufficiently large, and then integrate it by parts over $\mathbb{T}^3$ to get
\begin{align}\label{spatial-a}
  &\int_{\mathbb{T}^3}|\nabla \partial^{\gamma'}a|^2+K |\nabla \partial^{\gamma'} c|^2+(\nabla \partial^{\gamma'} a\cdot \nabla \partial^{\gamma'} c)+(\nabla \cdot\partial^{\gamma'} E)^2-\Delta \partial^{\gamma'}\phi \partial^{\gamma'}(e^\phi) dx\nonumber\\
  &\,\quad  \leq C\bigg(\|\partial^0 \partial^{\gamma'} b\|+\sum_i\|\partial^{\gamma'}(l_{bi}+h_{bi})\|+\|\partial^{\gamma'}(l_{c}+h_{c})\|\bigg)
  \left(\|\nabla \partial^{\gamma'} a\|+\|\nabla \partial^{\gamma'} c\|\right).
\end{align}
For the term $\Delta \partial^{\gamma'}\phi \partial^{\gamma'}(e^\phi)$ on the left-hand side of \eqref{spatial-a}, we have
\begin{align}
  -\int_{\mathbb{T}^3}\Delta \partial^{\gamma'}\phi \partial^{\gamma'}(e^\phi) dx
  =\,&\int_{\mathbb{T}^3} \nabla \partial^{\gamma'}\phi \cdot \partial^{\gamma'}(\nabla \phi e^\phi)dx\nonumber\\
= \,&\int_{\mathbb{T}^3}|\nabla \partial^{\gamma'}\phi|^2 e^\phi dx+\int_{\mathbb{T}^3}\nabla \partial^{\gamma'}\phi \cdot \mathcal{V} e^\phi dx,\label{star1}
\end{align}
where every component of $\mathcal{V}$ is a finite sum: each element therein is a product of at least two terms of $\partial^{\gamma''}\phi$, $0<|\gamma''|\leq|\gamma'|$. For example, if $|\gamma'|=1$, the corresponding $\mathcal{V}=\nabla \phi \partial^{\gamma'}\phi$; if
$|\gamma'|=2$, the corresponding $\mathcal{V}=\sum\partial^{i}\nabla\phi \partial^k \phi +\partial^{ij}\phi \nabla\phi+\partial^i \phi\partial^j \phi\nabla \phi$.
Recalling Proposition \ref{elliprop}, we have
$$
\|\partial^{\gamma''}\phi\|_{L^\infty(\mathbb{T}^3)}\leq C \|\partial^{\gamma''}\phi\|_{H^2(\mathbb{T}^3)}\leq C\sqrt{M_0}\ll 1.
$$
 Hence it holds
\begin{align}
  \int_{\mathbb{T}^3}\nabla \partial^{\gamma'}\phi \cdot\mathcal{V} e^\phi dx
  \leq \,&C\sqrt{M_0} \sum_{0<|\gamma''|\leq |\gamma'|}\int_{\mathbb{T}^3}|\nabla \partial^{\gamma'}\phi|| \partial^{\gamma''}\phi| e^\phi dx\nonumber\\
  \leq \,& C\sqrt{M_0} \int_{\mathbb{T}^3}|\nabla \partial^{\gamma'}\phi|^2 e^\phi dx
  +C\sqrt{M_0}\sum_{0<|\gamma''|\leq |\gamma'|} \int_{\mathbb{T}^3}| \partial^{\gamma''}\phi|^2 e^\phi dx.\label{star2}
\end{align}
Note that $\int_{\mathbb{T}^3}\partial^{\gamma''} \phi dx=0$, by Poincar\'{e}'s inequality and $e^\phi\sim 1$ (due to $\sup_{t\geq 0}\|\phi\|_{L^\infty(\mathbb{T}^3)}=O(1)$), it holds
\begin{align*}
  \int_{\mathbb{T}^3}| \partial^{\gamma''}\phi|^2 e^\phi dx\leq\,& C\|\partial^{\gamma''}\phi\|^2\leq C\|\nabla \partial^{\gamma''}\phi\|^2 \leq \cdots \leq C\|\nabla^{(|\gamma'|+1)} \phi\|^2\\
  \leq\, & C \sum_{|\varpi|=|\gamma'|}
  \int_{\mathbb{T}^3}| \nabla\partial^{\varpi}\phi|^2 e^\phi dx,
\end{align*}
which implies that
\begin{align}
  \sum_{0<|\gamma''|\leq |\gamma'|} \int_{\mathbb{T}^3}| \partial^{\gamma''}\phi|^2 e^\phi dx\leq C \sum_{|\varpi|=|\gamma'|}
  \int_{\mathbb{T}^3}| \nabla\partial^{\varpi}\phi|^2 e^\phi dx.\label{star3}
\end{align}
Hence we obtain from \eqref{star1}--\eqref{star3} that
\begin{align*}
  -\int_{\mathbb{T}^3}\Delta \partial^{\gamma'}\phi \partial^{\gamma'}(e^\phi) dx\,
  &\geq \int_{\mathbb{T}^3}|\nabla \partial^{\gamma'}\phi|^2 e^\phi dx-\left|\int_{\mathbb{T}^3}\nabla \partial^{\gamma'}\phi \cdot \mathcal{V} e^\phi dx\right|\\
  &\geq (1-C\sqrt{M_0}) \int_{\mathbb{T}^3}|\nabla \partial^{\gamma'}\phi|^2 e^\phi dx
  -C\sqrt{M_0} \sum_{0<|\gamma''|\leq |\gamma'|}\int_{\mathbb{T}^3}| \partial^{\gamma''}\phi|^2 e^\phi dx\\
  &\geq (1-C\sqrt{M_0}) \int_{\mathbb{T}^3}|\nabla \partial^{\gamma'}\phi|^2 e^\phi dx
  -C\sqrt{M_0} \sum_{|\varpi|=|\gamma'|}
  \int_{\mathbb{T}^3}| \nabla\partial^{\varpi}\phi|^2 e^\phi dx.
\end{align*}
Summing the above inequality over $0<|\gamma'|\leq N-1$ yields
\begin{align*}
  &\sum_{0<|\gamma'|\leq N-1}\int_{\mathbb{T}^3}-\Delta \partial^{\gamma'}\phi \partial^{\gamma'}(e^\phi) dx\\
  \geq\,& (1-C\sqrt{M_0}) \sum_{0<|\gamma'|\leq N-1}\int_{\mathbb{T}^3}|\nabla \partial^{\gamma'}\phi|^2 e^\phi dx
  -C\sqrt{M_0} \sum_{0<|\gamma'|\leq N-1}\sum_{|\varpi|=|\gamma'|}
  \int_{\mathbb{T}^3}| \nabla\partial^{\varpi}\phi|^2 e^\phi dx\\
  \geq\,& \frac{1}{2} \sum_{0<|\gamma'|\leq N-1}\int_{\mathbb{T}^3}|\nabla \partial^{\gamma'}\phi|^2 e^\phi dx.
\end{align*}
Therefore, if we take the sum of \eqref{spatial-a} over $0<|\gamma'|\leq N-1$, we deduce that
\begin{align}\label{spatial-a-sum}
  &\sum_{0<|\gamma'|\leq N-1}\int_{\mathbb{T}^3}|\nabla \partial^{\gamma'}a|^2+K |\nabla \partial^{\gamma'} c|^2+(\nabla \partial^{\gamma'} a\cdot \nabla \partial^{\gamma'} c)+(\nabla \cdot\partial^{\gamma'} E)^2
  + \frac{1}{2} |\nabla \partial^{\gamma'}\phi|^2 e^\phi dx\nonumber\\
  &\,\qquad  \leq C\sum_{0<|\gamma'|\leq N-1}\Big(\|\partial^0 \partial^{\gamma'} b\|+\sum_i\|\partial^{\gamma'}(l_{bi}+h_{bi})\|+\|\partial^{\gamma'}(l_{c}+h_{c})\|\Big)
  \sum_{0<|\gamma'|\leq N-1}\big(\|\nabla \partial^{\gamma'} a\|+\|\nabla \partial^{\gamma'} c\|\big).
\end{align}
It immediately follows from \eqref{spatial-a-sum} that (with $K\gg 1$)
\begin{align*}
  \sum_{0<|\gamma|\leq N-1}(\|\nabla \partial^\gamma a\|+\|\nabla \partial^\gamma c\|)
  \leq\,& C \sum_{0<|\gamma'|\leq N-1}\Big(\|\partial^0 \partial^{\gamma'} b\|+\sum_i\|\partial^{\gamma'}(l_{bi}+h_{bi})\|+\|\partial^{\gamma'}(l_{c}+h_{c})\|\Big)\\
  \leq\, & C \sum_{|\gamma|\leq N}\|\{\mathbf{I}-\mathbf{P}\}\partial^\gamma f(t)\| +C\sqrt{M_0}\sum_{|\gamma|\leq N}\|\partial^\gamma f(t)\|.
\end{align*}

If $\gamma'=0$, the deduction becomes simpler. In this case, \eqref{spatial-a} is exactly as
\begin{align}
  &\int_{\mathbb{T}^3}|\nabla a|^2+K |\nabla  c|^2+(\nabla  a\cdot \nabla c)+(\nabla \cdot E)^2+|\nabla \phi|^2 e^\phi dx\nonumber\\
  &\, \quad \leq C\bigg(\|\partial^0  b\|+\sum_i\|l_{bi}+h_{bi}\|+\|l_{c}+h_{c}\|\bigg)
  \left(\|\nabla  a\|+\|\nabla  c\|\right),
\end{align}
which implies that
\begin{align*}
  \|\nabla  a\|+\|\nabla  c\|\leq C \sum_{|\gamma|\leq N}\|\{\mathbf{I}-\mathbf{P}\}\partial^\gamma f(t)\| +C\sqrt{M_0}\sum_{|\gamma|\leq N}\|\partial^\gamma f(t)\|,
\end{align*}
and hence
\begin{align*}
  \sum_{0<|\gamma|\leq N}\|\partial^\gamma  a\|\leq C \sum_{|\gamma|\leq N}\|\{\mathbf{I}-\mathbf{P}\}\partial^\gamma f(t)\| +C\sqrt{M_0}\sum_{|\gamma|\leq N}\|\partial^\gamma f(t)\|.
\end{align*}
Furthermore,
by using Poincar\'{e}'s inequality, Lemma \ref{lemma-ac} and Proposition \ref{UH2}, we get
\begin{align*}
  \|a\|\leq C\bigg\{\|\nabla a\|+\bigg|\int_{\mathbb{T}^3} a dx\bigg|\bigg\}
  \leq   C \{\|\nabla a\|
  +\|\phi\|^2+\|\nabla_x \phi\|^2\}
  \leq   C \{\|\nabla a\|+\|f\|^2\}.
\end{align*}
In summary, $\sum_{|\gamma|\leq N}\|\partial^\gamma a\|$ is controlled by the right-hand side of \eqref{abc}.

\smallskip
\emph{Step 4. The estimate of $\partial^\gamma b(t,x)$, $\gamma=[\gamma_0,0,0,0]$, $|\gamma|\leq N$.}
By Lemma \ref{lemma-ac}, we have $\int_{\mathbb{T}^3} b dx=0$. It thus follows that $\int_{\mathbb{T}^3}\partial^\gamma b dx=0$. If $|\gamma|\leq N-1$, we use Poincar\'{e}'s inequality to get
\begin{align*}
 \|\partial^\gamma b\|\leq C \|\nabla \partial^\gamma b\|\leq C \sum_{|\gamma|\leq N}\|\{\mathbf{I}-\mathbf{P}\}\partial^\gamma f(t)\| +C\sqrt{M_0}\sum_{|\gamma|\leq N}\|\partial^\gamma f(t)\|.
\end{align*}

If $|\gamma|=N$, we take $\partial^{\gamma'}$ derivative $(|\gamma'|=N-1)$ on the identity \eqref{lbi} to get
\begin{align*}
  \partial^0 \partial^{\gamma'}b_i +\partial^i \partial^{\gamma'}a-\partial^{\gamma'}E_i=\partial^{\gamma'} (l_{bi}+h_{bi}).
\end{align*}
Thus,
\begin{align*}
  \sum_{|\gamma|=N}\|\partial^\gamma b\|\leq \sum_{|\gamma|=N}\|\partial^\gamma a\|+\sum_{|\gamma'|=N-1}\|\partial^{\gamma'}E\|+\sum_{|\gamma'|=N-1}\sum_i \|\partial^{\gamma'} (l_{bi}+h_{bi})\|.
\end{align*}
Recalling Proposition \ref{elliprop}, we know that
\begin{align*}
  \sum_{|\gamma'|=N-1}\|\partial^{\gamma'}E\|
  \leq \sum_{|\gamma|\leq N}\|\partial^{\gamma}\phi\|
  \leq C \sum_{|\gamma|\leq N}\left\|\partial^{\gamma}\int_{\mathbb{R}^3}f \sqrt{\mu} dv\right\|
  \leq  C\sum_{|\gamma|\leq N}\{\| \partial^{\gamma}a\|+\|\partial^{\gamma} c\|\}.
\end{align*}
Therefore, we obtain that
\begin{align*}
  \sum_{|\gamma|=N}\|\partial^\gamma b\|
  \leq\, & C\sum_{|\gamma|\leq N}\{\| \partial^{\gamma}a\|+\|\partial^{\gamma} c\|\}
  +\sum_{|\gamma'|=N-1}\sum_i \|\partial^{\gamma'} (l_{bi}+h_{bi})\|\\
  \leq\, &C \sum_{|\gamma|\leq N}\|\{\mathbf{I}-\mathbf{P}\}\partial^\gamma f(t)\| +C\sqrt{M_0}\sum_{|\gamma|\leq N}\|\partial^\gamma f(t)\|.
\end{align*}
In summary, $\sum_{|\gamma|\leq N}\|\partial^\gamma b\|$ is controlled by the right-hand side of \eqref{abc}.

Combining the results of the above four steps, the crucial estimate \eqref{abc} is verified. Thus, the proof of Theorem \ref{coercivityThm} is complete.
\end{proof}

\section{Global existence and exponential decay of classical solutions to the system
 \eqref{perturb eqn}--\eqref{poisson}}\label{decay-section}

In this section, we prove the global existence and exponential decay of classical solutions to
the ionic Vlasov-Poisson-Boltzmann system \eqref{perturb eqn}--\eqref{poisson}. We first establish a refined energy estimate of classical  solutions.

\begin{lemma}
Suppose that $f(t,x,v)$ is the unique solution to the system \eqref{perturb eqn}--\eqref{poisson} constructed in Theorem \ref{LocalTheorem}  satisfying the conservation laws \eqref{mass}--\eqref{energy}.
Let $\sum_{|\gamma|\leq N}\|\partial^\gamma f(t)\|^2\leq M_0\ll 1$ be valid. Then for any given $0\leq m\leq N$, $|\beta|\leq m$, there exist constants $C_{|\beta|}>0$, $\delta_m>0$ and $C_m>0$ such that
\begin{align}\label{estimate-m}
 \sum_{|\beta|\leq m, |\gamma|+|\beta|\leq N}\left\{\frac{d}{dt}\Big\{C_{|\beta|}\|\partial^\gamma_\beta f\|^2+\|\partial^\gamma \nabla \phi\|^2  +\big\|\partial^\gamma \phi e^{\frac{\phi}{2}}\big\|^2 \Big\}(t)+\delta_m \|\partial^\gamma_\beta f\|^2_{\nu}(t)
 \right\}\leq C_m \vvvert f \vvvert \vvvert f \vvvert_{\nu}^2 (t).
\end{align}
\end{lemma}

\begin{proof}
First, we take $\partial^\gamma$ $(|\beta|=0)$ derivatives of \eqref{perturb eqn} to get
\begin{align}\label{partial-x}
  &\{\partial_t +v\cdot \nabla_x+E\cdot \nabla_v\}\partial^\gamma f-\partial^\gamma E \cdot v\sqrt{\mu}+\mathcal{L}\{\partial^\gamma f\}\nonumber\\
  &\quad \, =-\sum_{\gamma_1\neq 0}C_\gamma^{\gamma_1}\partial^{\gamma_1}E\cdot \nabla_v \partial^{\gamma-\gamma_1}f
  +\sum_{\gamma_1}C_\gamma^{\gamma_1}\Big\{\partial^{\gamma_1}E\cdot \frac{v}{2}\partial^{\gamma-\gamma_1}f+\Gamma(\partial^{\gamma_1}f, \partial^{\gamma-\gamma_1}f)\Big\}.
\end{align}
When taking inner product of the above equality with $\partial^\gamma f$, we investigate the term
$  -\int\!\!\int_{\mathbb{T}^3\times \mathbb{R}^3}\partial^\gamma E \cdot v\sqrt{\mu} \partial^\gamma f dvdx$
and rewrite it as
\begin{align*}
  -\int\!\!\int_{\mathbb{T}^3\times \mathbb{R}^3}\partial^\gamma E \cdot v\sqrt{\mu} \partial^\gamma f dvdx
  =\,&\int_{\mathbb{T}^3} \partial^\gamma \nabla\phi \cdot \partial^\gamma \left(\int_{\mathbb{R}^3}v \sqrt{\mu}f dv\right)dx\\
  =\,& -\int_{\mathbb{T}^3} \partial^\gamma\phi \partial^\gamma \nabla\cdot (\rho u)dx\\
  =\,& \int_{\mathbb{T}^3} \partial^\gamma\phi \partial^\gamma \partial_t \rho dx\\
  =\,& \int_{\mathbb{T}^3} \partial^\gamma\phi \partial^\gamma (\partial_t \phi e^\phi-\partial_t \Delta \phi) dx\\
  =\,& \frac{1}{2}\frac{d}{dt}\left( \int_{\mathbb{T}^3} |\nabla \partial^\gamma \phi|^2 dx\right)+\int_{\mathbb{T}^3} \partial^\gamma\phi \partial^\gamma (\partial_t \phi e^\phi) dx\\
  =\,& \frac{1}{2}\frac{d}{dt}\left( \int_{\mathbb{T}^3} |\nabla \partial^\gamma \phi|^2 dx\right)+\frac{1}{2}\frac{d}{dt}\left( \int_{\mathbb{T}^3} | \partial^\gamma \phi|^2  e^\phi dx\right)
  +\int_{\mathbb{T}^3}\mathcal{W}e^\phi dx,
\end{align*}
where $\mathcal{W}$ is a finite sum: each element herein is a product of at least three terms of $\partial^{\tilde{\gamma}}\phi$, $0<|\tilde{\gamma}|\leq|\gamma|$. Recalling Proposition \ref{elliprop}, we thus have
\begin{align*}
  \int_{\mathbb{T}^3}\mathcal{W}e^\phi dx \leq C \vvvert f \vvvert \vvvert f \vvvert_\nu^2.
\end{align*}
Multiplying \eqref{partial-x} by $\partial^\gamma f$, by using Theorem \ref{coercivityThm} and the similar procedures as (iic) and (iid) in Section \ref{local-section}, we have
\begin{align*}
  \frac{d}{dt}\sum_{|\gamma|\leq N}\Big(\frac{1}{2}\|\partial^\gamma f\|^2+\|\partial^\gamma \nabla \phi\|^2  +\|\partial^\gamma \phi e^{\frac{\phi}{2}}\|^2   \Big)(t)+\delta_0 \sum_{|\gamma|\leq N}\|\partial^\gamma f\|^2_\nu(t)
  \leq C \sum_{|\gamma|\leq N}\vvvert f \vvvert \vvvert f \vvvert_{\nu}^2,
\end{align*}
which is \eqref{estimate-m} for $|\beta|=0$ case.

Second, assume that \eqref{estimate-m} is valid for $|\beta|\leq m$. We shall verify \eqref{estimate-m}
holds for the case $|\beta|=m+1$.
Take $\partial_\beta^\gamma$ derivatives of \eqref{perturb eqn}, it holds
\begin{align}
   \{\partial_t +v\cdot \nabla_x+E\cdot \nabla_v+\nu\}\partial^\gamma_\beta f
= \,& \partial^\gamma E \cdot \partial_\beta (v\sqrt{\mu})-\sum_{\beta_1\neq 0}C_\beta^{\beta_1}
  \{\partial_{\beta_1}v\cdot \nabla_x+\partial_{\beta_1} \nu\}\partial_{\beta-\beta_1}^\gamma f\nonumber\\
  &  + \sum C_\beta^{\beta_1} C_\gamma^{\gamma_1} \partial^{\gamma_1}E \cdot \partial_{\beta_1}\left(\frac{v}{2}\right)\partial_{\beta-\beta_1}^{\gamma-\gamma_1} f-\sum_{\gamma_1\neq 0}C_\gamma^{\gamma_1} \partial^{\gamma_1}E \cdot \nabla_v \partial^{\gamma-\gamma_1}_\beta f\nonumber\\
  &  + \partial_\beta [\mathcal{K}\partial^\gamma f]+\sum C_\gamma^{\gamma_1} \partial_\beta \Gamma (\partial^{\gamma_1}f, \partial^{\gamma-\gamma_1}f).\label{f-eqn}
\end{align}
We consider the inner product of \eqref{f-eqn} with $ \partial_\beta^\gamma f$. By direct calculations, the contribution of the term $\partial^\gamma E \cdot \partial_\beta (v\sqrt{\mu})$ in \eqref{f-eqn} is
\begin{align*}
  ( \partial^\gamma E \cdot \partial_\beta (v\sqrt{\mu}) , \partial_\beta^\gamma f)
  = (-1)^{|\beta|} ( \partial^\gamma E \cdot \partial^2_\beta (v\sqrt{\mu}) , \partial^\gamma f)
  \leq    C\|\partial^\gamma E\|\|\partial^\gamma f\|\leq C \sum_{|\gamma|\leq N}\|\partial^\gamma f\|^2.
\end{align*}
By \eqref{nu-derivative} in Lemma \ref{K} and the compact interpolation for $\partial_v^{\beta-\beta_1}$ between $\partial_v^\gamma$ $(|\gamma|=|\beta|)$ and $\partial_v^0$, the contribution of the term $\sum_{\beta_1\neq 0}C_\beta^{\beta_1}
  \{\partial_{\beta_1}v\cdot \nabla_x+\partial_{\beta_1} \nu\}\partial_{\beta-\beta_1}^\gamma f$ in \eqref{f-eqn} is
  $$\eta \|\partial_\beta^\gamma f\|^2+C_\eta \sum_{|\beta_1|=1}\|\nabla_x \partial^\gamma_{\beta-\beta_1}f\|^2.$$
By Lemma \ref{K} again, the contribution of the term $\partial_\beta [\mathcal{K}\partial^\gamma f]$ in \eqref{f-eqn} is
  $$\eta \sum_{|\beta'|=|\beta|}\|\partial_{\beta'}^\gamma f\|^2+C_\eta \|\partial^\gamma f\|^2.$$
By  simple modifications of the arguments in Section \ref{local-section}, the contribution of other terms on the right-hand side of \eqref{f-eqn} is $$C \vvvert f \vvvert \vvvert f \vvvert_{\nu}^2.$$
In summary, we deduce that
\begin{align}\label{m+1'}
   \frac{1}{2}\frac{d}{dt}\|\partial_\beta^\gamma f\|^2(t)+\|\partial_\beta^\gamma f\|_{\nu}^2(t)
  \leq \,& C\|\partial^\gamma E\|\|\partial^\gamma f\|+\eta \|\partial_\beta^\gamma f\|^2+C_\eta \sum_{|\beta_1|=1}\|\nabla_x \partial^\gamma_{\beta-\beta_1}f\|^2\nonumber\\
  &+ \eta \sum_{|\beta'|=|\beta|}\|\partial_{\beta'}^\gamma f\|^2+C_\eta \|\partial^\gamma f\|^2
   +C \vvvert f \vvvert \vvvert f \vvvert_{\nu}^2
\end{align}
for some sufficiently small constant $\eta>0$.

Summing \eqref{m+1'} over $|\beta|=m+1$ and $|\gamma|+|\beta|\leq N$ yields
\begin{align}\label{m+1}
 & \sum_{\substack{|\beta|= m+1\\ |\gamma|+|\beta|\leq N}}\left\{\frac{d}{dt}\|\partial_\beta^\gamma f\|^2(t)+\|\partial_\beta^\gamma f\|_{\nu}^2(t)\right\}
  \leq C_\eta \sum_{\substack{|\beta|= m\\ |\gamma|+|\beta|\leq N}}\| \partial^\gamma_{\beta}f\|^2+C_\eta \sum_{|\gamma|\leq N}\|\partial^\gamma f\|^2+C \vvvert f \vvvert \vvvert f \vvvert_{\nu}^2.
\end{align}
Multiplying \eqref{m+1} by a sufficiently small number and add it to \eqref{estimate-m} for $|\beta|\leq m$,
  we obtain that
\begin{align}\label{estimate-m+1}
 \sum_{\substack{|\beta|\leq m+1\\ |\gamma|+|\beta|\leq N}}\left\{\frac{d}{dt}\Big\{C_{|\beta|}\|\partial^\gamma_\beta f\|^2+\|\partial^\gamma \nabla \phi\|^2  +\big\|\partial^\gamma \phi e^{\frac{\phi}{2}}\big\|^2 \Big\}(t)+\delta_{m+1} \|\partial^\gamma_\beta f\|^2_{\nu}(t)
 \right\}\leq C_{m+1} \vvvert f \vvvert (t)\vvvert f \vvvert_{\nu}^2 (t),
\end{align}
which inductively concludes the proof of the lemma.
\end{proof}

\begin{proof}[Proof of Theorem \ref{Thm1} \emph{(}continued\emph{)}]

With \eqref{estimate-m} in hand, we now prove the global existence and exponential decay of the solution to \eqref{perturb eqn}--\eqref{poisson} with a bootstrap argument. Fix $M_0\ll 1$ such that Theorems \ref{LocalTheorem} and \ref{coercivityThm} hold. In the estimate \eqref{estimate-m}, we choose $m=N$, and then define
\begin{align*}
  y(t):=\, &\sum_{|\gamma|+|\beta|\leq N}\left\{C_{|\beta|}\|\partial_\beta^\gamma f\|^2+\|\partial^\gamma \nabla\phi\|^2+\|\partial^\gamma \phi e^{\phi/2}\|^2\right\}(t)\\
  \sim\, & \sum_{|\gamma|+|\beta|\leq N}\left\{C_{|\beta|}\|\partial_\beta^\gamma f\|^2+\|\partial^\gamma \nabla\phi\|^2+\|\partial^\gamma \phi \|^2\right\}(t)\\
  \sim\, & \sum_{|\gamma|+|\beta|\leq N} \|\partial_\beta^\gamma f\|^2(t)\\
  =\,& \vvvert f \vvvert^2(t),
\end{align*}
where the third line is obtained from the elliptic estimates in Proposition \ref{elliprop}.
Then \eqref{estimate-m} is equivalent to
\begin{align*}
  y'(t)+\delta_N \vvvert f \vvvert^2_{\nu}(t)\leq  C_{N} \vvvert f \vvvert \vvvert f \vvvert_{\nu}^2 (t)
  \leq   C\sqrt{M_0}\vvvert f \vvvert_{\nu}^2(t),
\end{align*}
which implies that
\begin{align}\label{y'(t)}
  y'(t)+\frac{\delta_N}{2} \vvvert f \vvvert^2_{\nu}(t)\leq 0.
\end{align}
It follows that
\begin{align}\label{y(t)}
  y(t)+\frac{\delta_N}{2}\int_0^t  \vvvert f \vvvert^2_{\nu} (s)ds\leq y(0)\leq C \vvvert  f_0\vvvert   ^2.
\end{align}
Recall the energy functional defined in \eqref{Efunctional}, i.e., $\mathcal{E}(t)= \vvvert f \vvvert^2(t)+\int_0^t \vvvert f \vvvert^2_{\nu}(s)ds$. Then \eqref{y(t)} indicates that
\begin{align*}
  \mathcal{E}(t)\leq C\vvvert  f_0\vvvert   ^2\leq C_* \mathcal{E}(0),
\end{align*}
for some $C_*>0$. For a constant $M\leq  {M_0}/({2C_*})$, we choose the initial datum $\mathcal{E}(0)\leq M$ and define
\begin{align*}
  T^*:= \sup_t \{t: \,\mathcal{E}(t)\leq M_0\}>0,
\end{align*}
then we have
\begin{align*}
  \sup_{0\leq t\leq T^*}\mathcal{E}(t)\leq C_* M<2C_* M=M_0,
\end{align*}
which contradicts the definition of $T^*$. Hence the maximal existence time $T^*$ can be extended to the infinity.

Moreover, by the fact that $y(t)\sim \vvvert f \vvvert^2(t)$, \eqref{y'(t)} indicates that
\begin{align*}
  y'(t)+Cy(t)\leq 0,
\end{align*}
for some $C>0$, and thus
\begin{align*}
  y(t)\leq e^{-Ct}y(0).
\end{align*}
Therefore, it holds
\begin{align*}
  \vvvert f \vvvert^2(t)\leq C e^{-Ct}\vvvert  f_0\vvvert^2.
\end{align*}
Hence, the whole proof of Theorem \ref{Thm1} is complete.
\end{proof}

%\hfill $\square$

%\smallskip\medskip
%{\bf Acknowledgements}:

\vspace{2mm}
		\noindent
	 \noindent \textbf{Acknowledgement.}
Li is  supported in part by  NSFC (Grant Nos. 12331007, 12071212)  and   the ``333 Project" of Jiangsu Province.
Wang is supported in part by NSFC (Grant No. 12401288)
and
the Jiangsu Provincial Research
Foundation for Basic Research (Grant No. BK20241259).

\vspace{2mm}
		\noindent
	 \noindent \textbf{Conflict of interest.} The authors do not have any possible conflicts of interest.

	\vspace{2mm}

	\noindent \textbf{Data availability statement.}
	 Data sharing is not applicable to this article, as no data sets were generated or analyzed during the current study.
	
%and the Jiangsu Provincial Scientific Research Center of Applied Mathematics (Grant No. BK20233002).
%{\bf and the Start-up Research Fund of Southeast University.}

%The work of the authors was supported by a National Natural Science Foundation of
% China (Grant No. ).

%{\bf Data Available Statement}: The data that support the findings of this study are available in public repository.


\begin{thebibliography}{99}

\bibitem{Degond}C. Bardos and P. Degond, Global existence for the Vlasov-Poisson equation in 3 space
variables with small initial data, \emph{Ann. Inst. H. Poincar\'{e} Anal. Non Lin\'{e}aire}, \textbf{2}(1985), 101--118.

\bibitem{Bardos} C. Bardos, F. Golse, T. T. Nguyen, and R. Sentis, The Maxwell-Boltzmann approximation for
ion kinetic modeling, \emph{Phys. D}, \textbf{376}/\textbf{377}(2018), 94--107.

\bibitem{Bouchut}F. Bouchut, Global weak solution of the Vlasov-Poisson system for small electrons mass, \emph{Comm. Partial
 Differential Equations}, \textbf{16}(1991), 1337--1365.






%\bibitem{Yunbai}Y.-B. Cao, Regularity of Boltzmann equation with external fields in convex domains of diffuse reflection, \emph{SIAM J. Math. Anal.}, \textbf{51}(2019), 3195--3275.


%\bibitem{Yunbai2}Y.-B. Cao and C. Kim, On some recent progress in the Vlasov-Poisson-Boltzmann system with diffuse reflection boundary, \emph{Recent advances in kinetic equations and applications}, Springer INdAM Ser., \textbf{48}(2021), 93--114.

\bibitem{Kim} Y.-B. Cao, C. Kim and D. Lee, Global strong solutions of the Vlasov-Poisson-Boltzmann system in bounded domains, \emph{Arch. Ration. Mech. Anal.}, \textbf{233}(2019), 1027--1130.


\bibitem{cercignani}C. Cercignani, R, Illner and  M. Pulvirenti, \emph{The Mathematical Theory of Dilute Gases}, Applied Mathematical Sciences, 106. {Springer-Verlag, New York, 1994}.


\bibitem{Iacobelli}L. Cesbron and M. Iacobelli, Global well-posedness of Vlasov-Poisson-type systems in
bounded domains, \emph{Anal. PDE}, \textbf{16}(2023), 2465--2494.


\bibitem{JAMS2023}S. Chaturvedi, J. Luk and T. T. Nguyen, The Vlasov-Poisson-Landau system in the weakly collisional regime, \emph{J. Amer. Math. Soc.}, \textbf{36}(2023), 1103--1189.


%\bibitem{Chen-Li2025JDE}J.-H. Chen and Y.-C. Li, 
%Long time asymptotic equivalence between the Vlasov-Poisson-Boltzmann equations and the isentropic compressible Navier-Stokes-Poisson equations,
%\emph{J. Differential Equations}, \textbf{433}(2025), 1--26.
%


%\bibitem{jsp}H.-X. Chen, C. Kim and Q. Li, Local well-posedness of Vlasov-Poisson-Boltzmann equation with generalized diffuse boundary condition, \emph{J. Stat. Phys.}, \textbf{179}(2), 535--631, 2020

%\bibitem{FFChen}F. F. Chen, \emph{Introduction to Plasma Physics and Controlled Fusion}, Springer International Publishing, Cham, 2016.


\bibitem{Choi-Koo-Song}Y. Choi, D. Koo and S. Song, Global existence of Lagrangian solutions to the ionic Vlasov-Poisson system, arXiv: 2501.13872.


\bibitem{Cordier-Grenier}S. Cordier and E. Grenier, Quasineutral limit of an
 Euler-Poisson system arising from plasma physics, \emph{Comm. Partial Differential Equations}, \textbf{25}(2000), 1099--1113.

%
%\bibitem{deng}D.-Q. Deng and R.-J. Duan, Spectral gap formation to kinetic equations with soft potentials in bounded domain, \emph{Comm. Math. Phys.}, \textbf{397}(2023), 1441--1489.
%

%\bibitem{Lions}R. DiPerna and P.-L. Lions, On the Cauchy problem for the Boltzmann equation: global existence and weak stability, \emph{Ann. Math.}, \textbf{130}, 312--366, 1989

%\bibitem{Ouyang}H.-J. Dong, Y. Guo and Z.-M. Ouyang, The Vlasov-Poisson-Landau system with the specular-reflection boundary condition, arXiv:2010.05314v2




%\bibitem{duan2017}R.-J. Duan, F.-M. Huang, Y. Wang and T. Yang, Global well-posedness of the Boltzmann equation with large amplitude initial data, \emph{Arch. Ration. Mech. Anal.}, \textbf{225}(1), 375--424, 2017

%\bibitem{duan2019}R.-J. Duan, F.-M. Huang, Y. Wang and Z. Zhang, Effects of soft interaction and non-isothermal boundary upon long-time dynamics of rarefied gas, \emph{Arch. Ration. Mech. Anal.}, 2019


%\bibitem{Duan-Liu} R. J. Duan and S. Q. Liu, Stability of the rarefaction wave of the Vlasov-Poisson-Boltzmann system, SIAM J. Math. Anal., 47 (2015), 3585-3647.

\bibitem{Duan2011}R.-J. Duan and R. Strain, Optimal time decay of the Vlasov-Poisson-Boltzmann system in $\mathbb{R}^3$, \emph{Arch. Ration. Mech. Anal.}, \textbf{199}(2011), 291--328.




\bibitem{Duan2012}R.-J. Duan, T. Yang and H.-J. Zhao, The Vlasov-Poisson-Boltzmann system in the whole space: the hard potential case, \emph{J. Differential Equations}, \textbf{252}(2012), 6356--6386.

\bibitem{Duan2013}R.-J. Duan, T. Yang and H.-J. Zhao, The Vlasov-Poisson-Boltzmann system for soft potentials, \emph{Math. Models Methods Appl. Sci.}, \textbf{23}(2013), 979--1028.

%\bibitem{Esposito}R. Esposito, Y. Guo, C. Kim and R. Marra, Non-isothermal boundary in the Boltzmann theory and Fourier law, \emph{Commun. Math. Phys.}, \textbf{323}(1), 177--239, 2013

\bibitem{Flynn} P. Flynn, Local well-posedness of the Vlasov-Poisson-Landau system and  related models, \emph{Kinet. Relat. Models}, \textbf{4}(2025), 583--608.

\bibitem{Gagnebin-Iacobelli}A. Gagnebin and M. Iacobelli,
Landau damping on the torus for the Vlasov-Poisson system with massless electrons, 
\emph{J. Differential Equations}, \textbf{376}(2023), 154--203.


%\bibitem{GT}D. Gilbarg and N. S. Trudinger, \emph{Elliptic Partial Differential Dquations of Second Order.
%Second edition}, Grundlehren der mathematischen Wissenschaften, Springer-Verlag, Berlin, 1983.

%\bibitem{Glassey}R. T. Glassey,  \emph{The Cauchy problem in kinetic theory}, Society for Industrial and Applied Mathematics (SIAM), Philadelphia, 1996.

\bibitem{ZhouFujun-SIAM-2021}W.-H. Gong, F.-J. Zhou and W.-J. Wu, 
Global strong solution and incompressible Navier-Stokes-Fourier-Poisson limit of the Vlasov-Poisson-Boltzmann system, \emph{SIAM J. Math. Anal.}, \textbf{53}(2021), 6424--6470.





\bibitem{Iacobelli2021a} M. Griffin-Pickering and M. Iacobelli, Global strong solutions in $\mathbb{R}^3$ for ionic Vlasov-Poisson systems, \emph{Kinet. Relat. Models}, \textbf{14}(2021), 571--597.

\bibitem{Iacobelli2021b} M. Griffin-Pickering and M. Iacobelli, Global well-posedness for the Vlasov-Poisson
 system with massless electrons in the 3-dimensional torus, \emph{Comm. Partial Differential Equations}, \textbf{46}(2021), 1892--1939.


%\bibitem{Guo1994}Y. Guo, Regularity of the Vlasov equations in a half space. \emph{Indiana. Math. J.}, \textbf{43}(1994), 255--320.

%\bibitem{Guo2001}Y. Guo, The Vlasov-Poisson-Boltzmann system near vacuum, \emph{Commun. Math. Phys.}, \textbf{218}(2), 293--313, 2001


\bibitem{Guo2002L}Y. Guo, The Landau equation in a period box, \emph{Comm. Math. Phys.}, \textbf{231}(2002), 391--434.


\bibitem{Guo2002}Y. Guo, The Vlasov-Poisson-Boltzmann system near Maxwellians, \emph{Comm. Pure Appl. Math.}, \textbf{55}(2002), 1104--1135.


\bibitem{Guo2003}Y. Guo, The Vlasov-Maxwell-Boltzmann system near Maxwellians, \emph{Invent. Math.}, \textbf{153}(2003), 593--630.


\bibitem{Guo-soft-2003}Y. Guo, Classical solutions to the Boltzmann equation for molecules with an angular cutoff, \emph{Arch. Ration. Mech. Anal.}, \textbf{169}(2003), 305--353.


\bibitem{Guo2010} Y. Guo, Decay and continuity of the Boltzmann equation in bounded domains, \emph{Arch. Ration. Mech. Anal.}, \textbf{197}(2010), 713--809.

\bibitem{Guo2012}Y. Guo, The Vlasov-Poisson-Landau system in a period box, \emph{J. Amer. Math. Soc.}, \textbf{25}(2012), 759--812.



%\bibitem{Guo2020}Y. Guo, H. J. Hwang, J. W. Jang and Z.-M. Ouyang, The Landau equation with the specular reflection boundary condition, \emph{Arch. Ration. Mech. Anal.}, \textbf{236}(3), 1389--1454, 2020

\bibitem{Juhi}Y. Guo and J. Jang, Global Hilbert expansion for the Vlasov-Poisson-Boltzmann system, \emph{Comm. Math. Phys.}, \textbf{299}(2010), 469--501.

%\bibitem{Guo2017} Y. Guo, C. Kim, D. Tonon, A. Trescases, Regularity of the Boltzmann equation in convex domains, \emph{Invent. Math.}, \textbf{207}(2017), 115--290.





\bibitem{He-Lei-Zhou}L.-B. He, Y.-J. Lei and Y.-L. Zhou, 
From the Vlasov-Poisson-Boltzmann system to the Vlasov-Poisson-Landau system with coulomb potential, \emph{SIAM J. Math. Anal.}, 
\textbf{55}(2023), 701--772.


%\bibitem{herau}F. H\'{e}rau, Introduction to hypocoercive methods and applications for simple linear inhomogeneous kinetic models, \emph{Lectures on the analysis of nonlinear partial differential equations}, Morningside Lect. Math., \textbf{5}(2018), 119--147.



\bibitem{Horst}E. Horst, On the asymptotic growth of the solutions of the Vlasov-Poisson system, \emph{Math.
Methods Appl. Sci.}, \textbf{16}(1993), 75--86.




\bibitem{Rendall}H. Hwang, A. Rendall and J. Vel\'{a}zquez, Optimal gradient estimates and asymptotic behavior
for the Vlasov-Poisson system with small initial data, \emph{Arch. Ration. Mech. Anal.}, \textbf{200}(2011),
313--360.

\bibitem{Hwang2009}H. J. Hwang and J. J. Vel\'{a}zquez, On global existence for the Vlasov-Poisson system in a half
space, \emph{J. Differential Equations}, \textbf{247}(2009), 1915--1948.

\bibitem{Hwang2010}H. J. Hwang and J. J. Vel\'{a}zquez, Global existence for the Vlasov-Poisson system in bounded
domains, \emph{Arch. Ration. Mech. Anal.}, \textbf{195}(2010), 763--796.

%\bibitem{Kim2011}C. Kim, Formation and propagation of discontinuity for Boltzmann equation in non-convex domains, \emph{Comm. Math. Phys.}, \textbf{308}(2011), 641--701.

\bibitem{fucai} F.-C. Li and Y.-C. Wang, Global strong solutions to the Vlasov-Poisson-Boltzmann system with soft potential in a bounded domain, \emph{J. Differential Equations}, \textbf{305}(2021), 143--205.

\bibitem{LW-23} F.-C. Li and  Y.-C.  Wang,  
  Global Euler-Poisson limit to the Vlasov-Poisson-Boltzmann system with soft potential,
\emph{SIAM J. Math. Anal.}, \textbf{55}(2023), 2877--2916.


\bibitem{LYZ-21}H.-L. Li, T. Yang and M.-Y. Zhong,
Diffusion limit of the Vlasov-Poisson-Boltzmann system, 
\emph{Kinet. Relat. Models}, \textbf{14}(2021), 211--255. 


\bibitem{Zhong2016}H.-L. Li, T. Yang and M.-Y. Zhong,
Spectrum analysis and optimal decay rates of the bipolar Vlasov-Poisson-Boltzmann equations,
\emph{Indiana Univ. Math. J.},  \textbf{65}(2016), 665--725.


\bibitem{Zhong2020}H.-L. Li, T. Yang and M.-Y. Zhong,
Green's function and pointwise space-time behaviors of the Vlasov-Poisson-Boltzmann system,
\emph{Arch. Ration. Mech. Anal.},  \textbf{235}(2020), 1011--1057.



\bibitem{Zhong2021}H.-L. Li, T. Yang and M.-Y. Zhong,
Spectrum analysis for the Vlasov-Poisson-Boltzmann system,
\emph{Arch. Ration. Mech. Anal.},  \textbf{241}(2021), 311--355.



\bibitem{Perthame}P. L. Lions and B. Perthame, Propagation of moments and regularity for the 3-dimensional
Vlasov-Poisson system, \emph{Invent. Math.}, \textbf{105}(1991), 415--430.


\bibitem{Mouhot-Villani}C. Mouhot, and C. Villani, 
On Landau damping, 
\emph{Acta Math.}, \textbf{207}(2011), 29--201.


\bibitem{Pfaffelmoser}K. Pfaffelmoser, Global classical solutions of the Vlasov-Poisson system in three dimensions
for general initial data, \emph{J. Differential Equations}, \textbf{95}(1992), 281--303.

\bibitem{Phillips}A. C. Phillips, \emph{The Physics of Stars} (\emph{second edition}), Wiley, 2013.

\bibitem{Schaeffer}J. Schaeffer, Global existence of smooth solutions to the Vlasov-Poisson system in three
dimensions, \emph{Comm. Partial Differential Equations}, \textbf{16}(1991), 1313--1335.

\bibitem{sentis} R. Sentis,  \emph{ Mathematical Models and Methods for Plasma Physics. Vol. 1.
Fluid Models},   Birkh\"{a}user/Springer, Cham, 2014.

\bibitem{Smulevici} J. Smulevici, Small data solutions of the Vlasov-Poisson system and the vector field method,
\emph{Ann. PDE}, \textbf{2}(2016), 1--55.

\bibitem{plasma-astro}B.-L. Tan, J.-T. Su, Y.-T. Li, J. Huang and C.-M. Tan, \emph{Plasma Astrophysics}
 (in Chinese), Beijing, Science Press, 2024.

\bibitem{TK} A. W. Trivelpiece and N.A. Krall, \emph{Principles of Plasma Physics}, McGraw-Hill, Englewood Cliffs, 1973.

%\bibitem{Villani}C. Villani, Hypocoercivity. \emph{Mem. Amer. Math. Soc.}, \textbf{202}(2009), 1--141.


%\bibitem{Lee}C. Kim and D. Lee, The Boltzmann equation with specular boundary condition in convex domains, \emph{Comm. Pure Appl. Math.}, \textbf{71}(3), 411--504, 2018


%\bibitem{Hwang} J. Kim, Y. Guo and H. J. Hwang, An $L^2$ to $L^\infty$ framework for the Landau equation, \emph{Peking Mathematical Journal}, \textbf{3}, 131--202, 2020


%\bibitem{Lieb} E.-H. Lieb, M. Loss, \emph{Analysis}, \emph{volume 14 of Graduate Studies in Mathematics}, 2nd ed. American Mathematical Society, Providence 2001

%\bibitem{shuangqian} S.-Q. Liu and X.-F. Yang, The initial boundary value problem for the Boltzmann equation with soft potential, \emph{Arch. Ration. Mech. Anal.}, \textbf{233}(1), 463--541, 2017

%\bibitem{Mischler}S. Mischler, On the initial boundary value problem for the Vlasov-Poisson-Boltzmann system, \emph{Commun. Math. Phys.}, \textbf{210}(2), 447--466, 2000

%\bibitem{SrainVMB}R. Strain,  The Vlasov-Maxwell-Boltzmann system in the whole space, \emph{Commun. Math. Phys.}, \textbf{268}(2), 543--567, 2006

%\bibitem{Strain}R. Strain and Y. Guo, Exponential decay for soft potentials near Maxwellians, \emph{Arch. Ration. Mech. Anal.}, \textbf{187}(2), 287--339, 2008

%\bibitem{WuZhuoqun}Z.-Q. Wu, J.-X. Yin and C.-P. Wang, \emph{Elliptic and parabolic equations}, World Scientific Publishing Co. Pte. Ltd., 2006.


\bibitem{xxz-14} Q.-H. Xiao, L.-J. Xiong and H.-J. Zhao, The Vlasov-Poisson-Boltzmann system for non-cutoff hard potentials, \emph{Sci. China Math.},  \textbf{57}(2014), 515--540.

\bibitem{Zhao2017}Q.-H. Xiao, L.-J. Xiong and H.-J. Zhao, The Vlasov-Poisson-Boltzmann system for the whole range of cutoff soft potentials, \emph{J. Funct. Anal.}, \textbf{272}(2017), 166--226.

\bibitem{Hongjun}T. Yang, H.-J. Yu and H.-J. Zhao, Cauchy problem for the Vlasov-Poisson-Boltzmann system. \emph{Arch. Ration. Mech. Anal.},  \textbf{182}(2006), 415--470.

\bibitem{Zhao2006}T. Yang and H.-J. Zhao, Global existence of classical solutions to the Vlasov-Poisson-Boltzmann system, \emph{Commun. Math. Phys.}, \textbf{268}(2006), 569--605.






















\end{thebibliography}
\end{document}